\documentclass[twoside]{paper}
\usepackage[T1]{fontenc}
\usepackage[latin9]{inputenc}
\usepackage{geometry}
\usepackage{color}
\usepackage{amsmath}
\usepackage{amsthm}
\usepackage{amssymb}
\usepackage{stackrel}
\usepackage{graphicx}
\usepackage{setspace}
\usepackage{esint}
\usepackage[unicode=true,pdfusetitle,
 bookmarks=true,bookmarksnumbered=false,bookmarksopen=false,
 breaklinks=false,pdfborder={0 0 1},backref=false,colorlinks=true]
 {hyperref}
\hypersetup{
 citecolor=blue, urlcolor=blue, linkcolor=red}

\makeatletter

\providecommand{\tabularnewline}{\\}

\numberwithin{equation}{section}
\numberwithin{figure}{section}
\numberwithin{table}{section}
\usepackage{enumitem}		
 \theoremstyle{definition}
 \newtheorem*{defn*}{\protect\definitionname}
  \theoremstyle{plain}
  \newtheorem*{prop*}{\protect\propositionname}
  \theoremstyle{remark}
  \newtheorem{rem}{\protect\remarkname}[section]
  \theoremstyle{definition}
  \newtheorem{defn}{\protect\definitionname}[section]
  \theoremstyle{plain}
  \newtheorem{thm}{\protect\theoremname}[section]
  \theoremstyle{plain}
  \newtheorem{prop}{\protect\propositionname}[section]
  \theoremstyle{plain}
  \newtheorem{lem}{\protect\lemmaname}[section]

\usepackage{amssymb, latexsym, amsmath, amscd, epsfig, color}

\newtheorem{obs}{Observation}[section]

\newcommand{\gr}{\mathcal{G}}
\newcommand{\moduli}{\mathcal{M}}
\newcommand{\tei}{\mathcal{T}}

\newcommand{\real}{\mathbb{R}}

\makeatother

  \providecommand{\definitionname}{Definition}
  \providecommand{\lemmaname}{Lemma}
  \providecommand{\propositionname}{Proposition}
  \providecommand{\remarkname}{Remark}
\providecommand{\theoremname}{Theorem}

\begin{document}

\title{Pressure Type Metrics on Spaces of Metric Graphs}

\author{Lien-Yung Kao\thanks{Department of Mathematics, University of Notre Dame, Notre Dame, IN
46545 USA. \textit{E-mail}\texttt{:\protect\href{mailto:lkao@nd.edu}{lkao@nd.edu}}}}

\maketitle
\begin{abstract}
\textsf{In this note, we consider two Riemannian metrics on a moduli
space of metric graphs. Each of them could be thought of as an analogue
of the Weil-Petersson metric on the moduli space of metric graphs.
We discuss and compare geometric features of these two metrics with
the ``classic'' Weil-Petersson metric in Teichmüller theory. This
paper is motivated by Pollicott and Sharp's work \cite{Pollicott:2014fs}.
Moreover, we fix some errors in \cite{Pollicott:2014fs}.}
\end{abstract}

\tableofcontents{}

\section{Introduction}

\textsf{This note is a further study of a dynamical-system-theoretically
defined metric on deformation spaces-- the pressure metric. The study
of pressure metrics is ignited by McMullen's study of the Weil-Petersson
metric on Teichmüller space. In \cite{McMullen:2008eh}, McMullen
proved that one can realize the Weil-Petersson metric on Teichmüller
space by the pressure metric (on a certain functional space). The
pressure metric has been a great tool for defining and studying the
Weil-Petersson metrics on a variety of contexts: Teichmüller spaces
\cite{McMullen:2008eh}, Anosov representations \cite{Bridgeman:2013to}
and Blaschke products \cite{McMullen:2008eh}. However, at this point,
most of our understandings of the geometry of the pressure metric
are coming from its relation with the ''classic'' Weil-Petersson
metric. In this work, carrying over ideas from Pollicott--Sharp's
work \cite{Pollicott:2014fs}, we focus on investigating the pressure
metric geometry from a more dynamical approach.}

\textsf{More explicitly, we follow Pollicott-Sharp's construction
of pressure type metrics on the moduli space of metric graphs, and
we consider two ``natural''}\footnote{\textsf{They are natural in the sense that the constructions of these
two metrics are close to the construction of the ``classic'' Weil-Petersson
metric on Teichmüller spaces introduced by McMullen.}}\textsf{ pressure type metrics on the moduli space of graphs. We correct
a formula in \cite{Pollicott:2014fs}, and using the revised formulas
we examine geometric features of these two ``natural'' pressure
type metrics on a moduli space of ``typical''}\footnote{\textsf{Graphs whose the fundamental group is a free group of rank
2. }}\textsf{ graphs, and these examples show that geometric behaviors
of these two ``natural'' pressure type metrics are very far from
what we know about the ``classic'' Weil-Petersson metric (i.e. the
pressure metric) on Teichmüller spaces. }

\textsf{To state our result more precisely and to put it in context,
we first review basic setups in Teichmüller theory (cf. Section \ref{sec:Preliminaries}
for more details). Let $S$ be a compact topological surface with
negative Euler characteristics. Teichmüller space $\mathcal{T}(S)$
could be thought of as the set of isotopy classes of Riemannian metrics
with constant curvature $-1$, and the moduli space $\mathcal{M}(S)$
could be described as the set of isometry classes of Riemannian metrics
with constant curvature $-1$. Moreover, the moduli space $\mathcal{M}(S)$
is obtained by quotienting Teichmüller space $\mathcal{T}(S)$ by
the mapping class group $\mathrm{MCG}(S)$. The Weil-Petersson metric
is a naturally defined and well-studied $\mathrm{MCG}$-invariant
metric on Teichmüller space (thus on the moduli space) with several
striking features:}
\begin{itemize}
\item \textsf{the Weil-Petersson metric is negatively curved,}
\item \textsf{the sectional curvature are neither bounded away form 0 nor
$-\infty$, and}
\item \textsf{the Weil-Petersson metric is incomplete.}
\end{itemize}
\textsf{McMullen's result in \cite{McMullen:2008eh} shows that on
Teichmüller space we can define a Riemannian metric, }\textit{the
pressure metric}\textsf{, via the thermodynamic formalism, and which
is exactly the Weil-Petersson metric. In other words, the pressure
metric shares these notable geometric features with the Weil-Petersson
metric on Teichmüller space. }

\begin{defn*}
Given a \textit{undirected finite graph} $\mathcal{G}$ with \textit{edge
set} $\mathcal{E}$. The \textit{edge weighting function} $l:\mathcal{E}\to\real_{>0}$
assigns to each edge a length, which endows a metric structure onto
$\mathcal{G}$. We call the pair $(\mathcal{G},l)$ a \textit{metric
graph}. 
\end{defn*}

\textsf{From a dynamical point of view, the metric graphs possess
very similar structures as Riemann surfaces. Dynamics of paths on
metric graphs is analogous to the dynamics of the geodesic flow for
Riemann surfaces. It is because the length weighting function $l$
on $\mathcal{G}$ plays the same role as a Riemannian metric on surfaces.
Hence, it is natural to begin the study of the pressure metric geometry
from deformation spaces of metric graphs. Here, our deformation space
corresponding to the graph $\mathcal{G}$ is the space $\mathcal{M}_{\mathcal{G}}$
of all edge weighting functions.}

\begin{defn*}
For a graph $\mathcal{G}$ and an edge weighting function $l$ the
\textit{entropy} $h(l)$ is defined by 
\[
h(l)=\lim_{T\to\infty}\frac{1}{T}\log\#\{\gamma;l(\gamma)<T\}
\]
 where $\gamma=(e_{0},e_{1},...,e_{n}=e_{0})$ is a closed cycle of
edges in $\mathcal{G}$ (without backtracking) and $l(\gamma)={\displaystyle \sum_{i=0}^{n-1}l(e_{i})}$. 
\end{defn*}

\textsf{From the dynamical perspective, the entropy $h(l)$ for metric
graphs, as we have seen in surface cases, is an important and informative
quality. We recall that the moduli space $\mathcal{M}(S)$ is the
collection (up to isometry) of Riemannian metric on $S$ with constant
curvature $-1$. We notice that the constant negative curvature condition
of $\mathcal{M}(S)$ could be interpreted dynamically by using the
constant (topological) entropy (of the geodesic flow on $S$) condition.
More precisely, because when $S$ has a constant negative curvature
(say $K(S))$, the topological entropy of the geodesic flow for $S$
is equal to $\sqrt{|K(S)|}$. Thus, to derive a close analogy to the
moduli space $\mathcal{M}(S)$, it is natural and dynamical meaningful
to consider the condition that entropy $h(l)$ equal to 1. Moreover,
there is one more reason for us to concentrate on the space $\mathcal{M}_{\mathcal{G}}^{1}$;
however, this reason is more technical and it's from the nature of
pressure type metrics. We will explain this reason in Remark \ref{Rmk: why h=00003D1 is enough}.
In what follows, we will focus on the deformation space $\mathcal{M}_{\mathcal{G}}^{1}$
of metric graphs with entropy 1. i.e. }

\textsf{
\[
\mathcal{M}_{\mathcal{G}}^{1}=\left\{ l:\mathcal{E}\to\real_{>0};h(l)=1\right\} .
\]
}

\textsf{Inspired by McMullen \cite{McMullen:2008eh}, Pollicott and
Sharp constructed a pressure type metric (they call it a Weil-Petersson
type metric) for metric graphs \cite{Pollicott:2014fs}. We follow
their work and  construct another pressure type metric for $\mathcal{M}_{\mathcal{G}}^{1}$.
We call these two pressure type metrics the }\textit{pressure metric}\textsf{
and the }\textit{Weil-Petersson metric}\textsf{ for $\mathcal{M}_{\mathcal{G}}^{1}$,
and denote them by $||\cdot||_{P}$ and $||\cdot||_{WP}$, respectively.
The Weil-Petersson metric $||\cdot||_{WP}$ for $\mathcal{M}_{\mathcal{G}}^{1}$
is indeed conformal to the pressure metric $||\cdot||_{P}$ and the
scaling function could be thought of a volume term of $\mathcal{G}$.
Through normalizing the pressure metric $||\cdot||_{P}$ by the volume
term of $\mathcal{G}$ , the resulting metric $||\cdot||_{WP}$ is
a closer analogy to McMullen's definition of the Weil-Petersson metric
on Teichmüller space and the definition of the pressure metric on
Hitchin components (cf. Theorem \ref{thm: pressue =00003D wp} and
Section\ref{sec:Pressure-metrics-on-graphs} for more details). Moreover,
in Section \ref{sub: pressure type on gh and formula}, we propose
several formulas and properties of pressure type metrics, which help
to illustrate the usefulness of the definition.}

\textsf{Using formulas that we get in Section \ref{sub:From-undirected-graphs},
we discuss four typical metric graphs: a figure 8 graph, a belt buckle,
a dumbbell, and a three-petal rose. The first three graphs share the
same fundamental group $F_{2}$ (the free group of rank 2). This makes
an interesting connection with the }\textit{outer space}\textsf{ \cite{Culler:1986jm}
in rank 2. For brevity, throughout this note, we denote the figure
8 graph, belt buckle, dumbbell, and the three-petal rose by $\mathcal{G}_{1}$,
$\mathcal{G}_{2}$, $\mathcal{G}_{3}$, and $\mathcal{G}_{4}$, respectively.
We summarize several results in Section \ref{sec:Examples} in the
following. }

\begin{prop*}
[Proposition \ref{prop: fig 8 press incomplete}, \ref{prop: fig 8 WP complete },  \ref{prop: Belt buckle Pressure}, \ref{prop: Dumbbell Pressure}, and Observation \ref{obs: Belt buckle WP}, \ref{obs: Dumbbell WP}, \ref{obs: 3-rose Pressure},
\ref{obs: 3-rose-WP}]-

\begin{enumerate}

\item There exist examples of graphs for which the metric $||\cdot||_{P}$
is not complete.

\item There exist examples of graphs for which the curvature of the
metric $||\cdot||_{P}$ is positive.

\item There exist examples of graphs for which the metric $||\cdot||_{WP}$
is complete.

\item There exist examples of graphs for which the curvature of the
metric $||\cdot||_{WP}$ take positive and negative values.

\end{enumerate}
\end{prop*}

\begin{center}

\begin{tabular}{|c|c|c|c|c|}
\hline 
 & \includegraphics[scale=0.3]{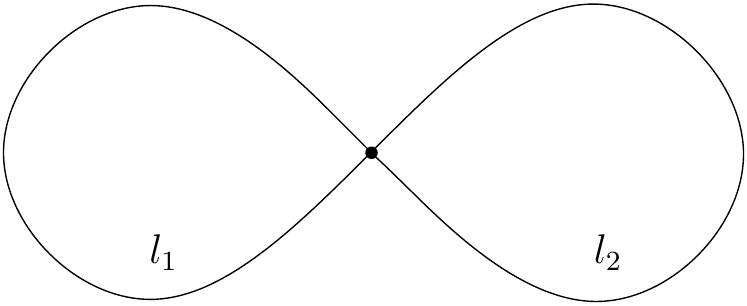} & \includegraphics[scale=0.3]{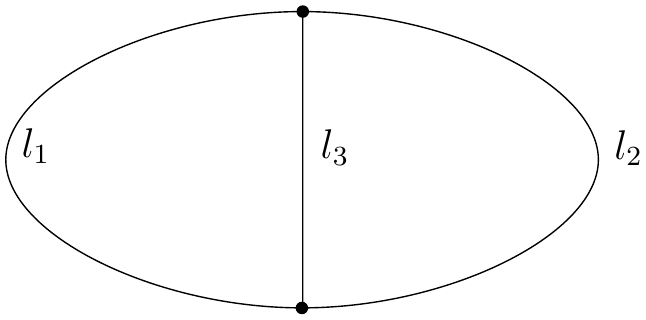} & \includegraphics[scale=0.4]{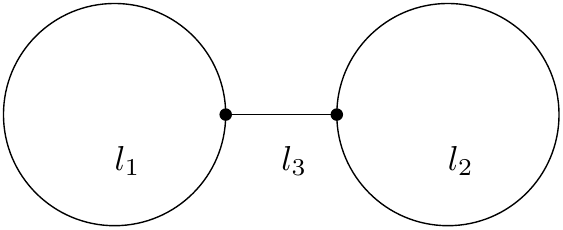} & \includegraphics[scale=0.2]{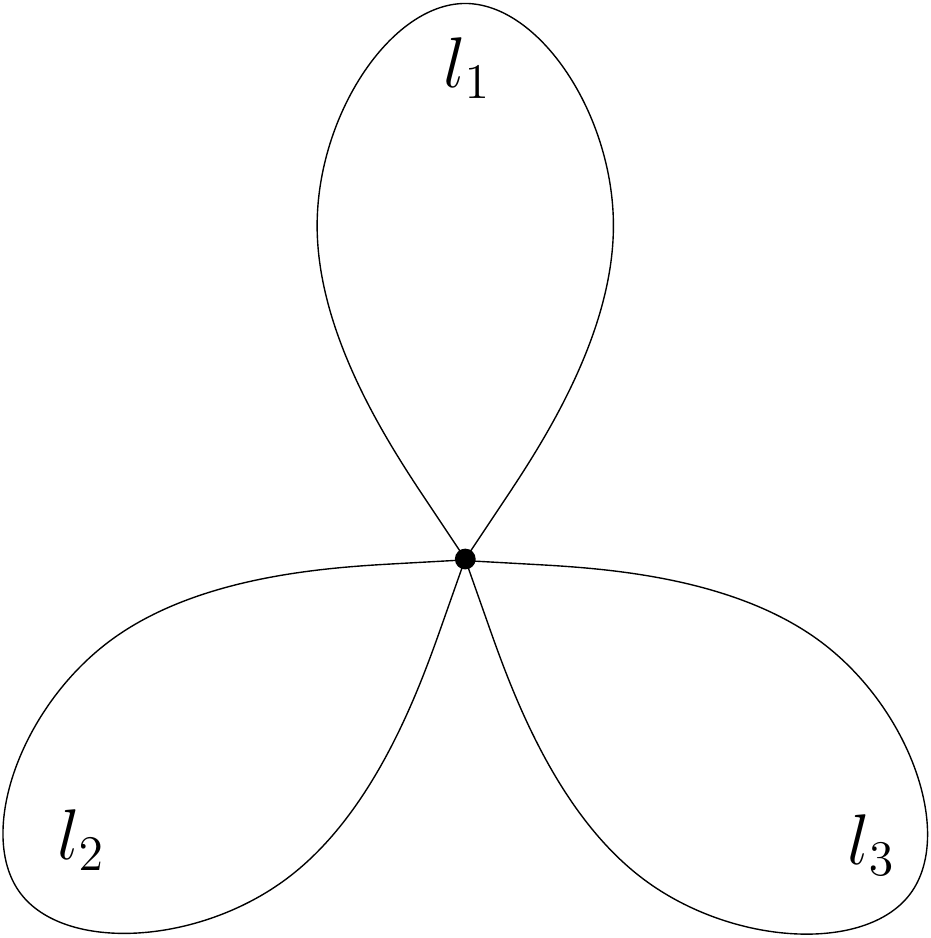}\tabularnewline
\hline 
 & Figure 8 & Belt buckle & Dumbbell & Three-petal rose\tabularnewline
\hline 
\hline 
$||\cdot||_{P}$ & incomplete & incomplete & incomplete & incomplete\tabularnewline
 &  & $0<C_{2}\leq K\leq C_{1}$ & $0<K$, unbounded & $0<C_{2}<K<C_{1}$\tabularnewline
\hline 
$||\cdot||_{WP}$ & complete & $C_{2}\leq K\leq C_{1}<0$ & $K$ takes $+$ and $-$ & $K$ takes $+$ and $-$\tabularnewline
\hline 
\end{tabular}

\end{center}

where $K$ denotes the Gaussian curvature.

\begin{rem}
\begin{enumerate}-

\item One need to be careful that our pressure metric $||\cdot||_{P}$
is the Weil-Petersson type metric defined in \cite{Pollicott:2014fs}.

\item One should compare these results with results in Pollicott-Sharp's
paper \cite{Pollicott:2014fs}. Although, Pollicott-Sharp and we are
working on the same examples, the results are not quite match. It
is because Lemma 3.3 in \cite{Pollicott:2014fs} (the main formula
for calculating the pressure metric used in \cite{Pollicott:2014fs}
) is true only if we input an extra condition that $f$ is normalized
with respect to the transfer operator (i.e. $\mathcal{L}_{f}1=1$).
For more details, one can check to our Remark \ref{rem: mistake in PS}.

\item The difference between \textbf{Proposition}s and \textbf{Observation}s
is that for \textbf{Proposition}s we give proofs and for \textbf{Observation}s
we give computer computation evidences. The reason that we skip proofs
for \textbf{Observation}s is because computations of the curvature
with respect to $||\cdot||_{WP}$ and $||\cdot||_{P}$ are not hard
but tedious and complex, which make no sense to analyze them in great
detail rather than using the help from computers.

\item It is surprising that for all of our examples the sectional
curvatures with respect to the pressure metric $||\cdot||_{P}$ are
positive.

\end{enumerate}
\end{rem}

\textsf{From the above proposition, we can see that both the pressure
metric and the Weil-Petersson metric for metric graphs shares some
common features with the Weil-Petersson metric on Teichmüller space.
However, through working on these explicit examples, we can conclude
that, from a pure dynamical setting, the pressure type metrics for
metric graphs don't behave like the Weil-Petersson metric on Teichmüller
space. In other words, the topology and the geometry of closed surfaces
force the pressure metric to be the Weil-Petersson metric on Teichmüller
space, but for graphs, without these special structures, the pressure
type metrics cannot reflect these well-known features in the Weil-Petersson
geometry.}

\textsf{This note is organized as the following. In the sake of completeness,
we will review from basic knowledge of the symbolic dynamics, thermodynamic
formalism and Teichmüller theory in Section \ref{sec:Preliminaries}.
In Section \ref{sec:Pressure-metrics-on-graphs}, we relate the deformation
space of metric graphs with a subshift of finite type and construct
the pressure type metrics for metric graph via this relation. Furthermore,
we give several useful formulas of them. In the last section, we discuss
the pressure (type) metric geometry through working on explicit examples. }

\subsection*{Acknowledgments}

\textsf{The author is extremely grateful to his Ph.D. advisor Prof.
François Ledrappier and Prof. Mark Pollicott. They lead the author
into the beautiful world of ergodic theory and dynamical systems,
especially the thermodynamic formalism. This work would have never
been possible without their support, guidance, and sharing of their
insightful ideas. }

\section{Preliminaries\label{sec:Preliminaries}}

\subsection{Symbolic dynamics and the thermodynamic formalism}

\textsf{We begin by recalling definitions and facts in symbolic dynamical
systems and the thermodynamic formalism. An excellent reference on
these two topics (concerning our approach) is the book \cite{Parry:1990tn}
written by Parry and Pollicott. }

\textsf{A $k\times k$ matrix $A$ is called }\textit{irreducible}\textsf{,
if for each pair $(i,j)$, $1\leq i,j\leq k$, there exists $n\geq1$
such that $A^{n}(i,j)>0$, where $A^{n}$ is an $n-$fold product
of $A$ with itself. Suppose $A$ is an irreducible matrix consisting
of $0$ and $1$, we define 
\[
\Sigma_{A}^{+}=\left\{ \underline{x}=(x_{n})_{n\in\mathbb{N}}\in\{1,2,..,k\}^{\mathbb{N}};A(x_{n},x_{n+1})=1\right\} .
\]
We consider the shift map $\sigma:\Sigma_{A}^{+}\to\Sigma_{A}^{+}$
by $\left(\sigma(\underline{x})\right)_{n}=x_{n+1}$, and then call
$(\Sigma_{A}^{+},\sigma)$ a }\textit{(one-sided) subshift of finite
type}\textsf{. We notice that $\Sigma_{A}^{+}$ is a compact zero
dimensional space with respect to the Tychonoff product topology,
and we can endow a metric $d$ on $\Sigma_{A}^{+}$ . More precisely,
this space is compact with respect to the metric 
\[
d(\underline{x},\underline{y})={\displaystyle \sum_{n=0}^{\infty}\frac{1-\delta(x_{n},y_{n})}{2^{n}}},
\]
where where $\delta(i,j)$ is the standard Kronecker delta.}
\begin{rem}
$A$ is irreducible implies that $(\Sigma_{A}^{+},\sigma)$ is (topological)
transitive (i.e. $\sigma:\Sigma_{A}^{+}\to\Sigma_{A}^{+}$ has a dense
orbit). 
\end{rem}
\textsf{We denote the set of continuous functions on $\Sigma_{A}^{+}$
by $C(\Sigma_{A}^{+})$. A function $f:\Sigma_{A}^{+}\to\real$ is
called }\textit{$\alpha-$Hölder continuous}\textsf{ if there exists
$C>0$ and $\alpha\in(0,1]$ such that for all $\underline{x},\underline{y}\in\Sigma_{A}^{+}$
we have $\left|f(\underline{x})-f(\underline{y})\right|\leq Cd(\underline{x},\underline{y})^{\alpha}$,
and we denote by $C^{\alpha}(\Sigma_{A}^{+})$ the space of $\alpha-$Hölder
continuous functions on $\Sigma_{A}^{+}$. We call a function $f:\Sigma_{A}^{+}\to\real$
is }\textit{Hölder continuous}\textsf{ if it is }\textit{$\alpha-$Hölder
continuous}\textsf{ for some $\alpha$. Two function $f,g\in C(\Sigma_{A}^{+})$
are called }\textit{Livšic cohomologous}\textsf{ ($f\sim g$), if
there exists $h\in C(\Sigma_{A}^{+})$ such that 
\[
f-g=h\circ\sigma-h.
\]
}

\subsubsection{Transfer operator, pressure and variance}

\textsf{Most of the following definitions could be generalized to
continuous functions $C(\Sigma_{A}^{+})$ on $\Sigma_{A}^{+}$. However,
in this note we don't need such a generality, so we will only focus
on Hölder continuous functions on $\Sigma_{A}^{+}$. }
\begin{defn}
The \textit{transfer operator (or Ruelle operator)} $\mathcal{L}_{f}:C(\Sigma_{A}^{+})\to\real$
of a Hölder continuous function $f$ is defined by 
\[
(\mathcal{L}_{f}w)(\underline{x})=\sum_{\underline{y}=\sigma^{-1}(\underline{x})}e^{f(\underline{y})}w(\underline{y}).
\]

\end{defn}

\begin{thm}
[Ruelle-Perron-Frobenius Theorem, Theorem 2.2 \cite{Parry:1990tn}]
\label{thm: RPF}Let $f$ be a Hölder continuous function on $\Sigma_{A}^{+}$
and suppose $A$ is irreducible. \begin{enumerate}

\item There is a simple maximal positive eigenvalue $\beta_{f}$
of $\mathcal{L}_{f}:C(\Sigma_{A}^{+})\to C(\Sigma_{A}^{+})$ with
a corresponding strictly positive eigenfunction $v_{f}\in C^{\alpha}(\Sigma_{A}^{+}),$
and $\beta_{f}$ realizes the spectral radius of $\mathcal{L}_{f}:C(\Sigma_{A}^{+})\to C(\Sigma_{A}^{+})$.

\item There exists a unique probability measure $\mu_{f}$ such that
$\mathcal{L}_{f}^{*}\mu_{f}=\beta_{f}\mu_{f}$. i.e. 
\[
\int\mathcal{L}_{f}\phi\mathrm{d}\mu_{f}=\beta_{f}\int\phi\mathrm{d}\mu_{f}\mbox{ for all }\phi\in C(\Sigma_{A}^{+}).
\]

\end{enumerate}
\end{thm}

\textsf{By the Ruelle-Perron-Frobenius theorem, we are now ready to
define the pressure. }

\begin{defn}
For every $f\in C^{\alpha}(\Sigma_{A}^{+})$, the \textit{pressure
of $f$,} $P(f)$, is the log of the spectral radius of $\mathcal{L}_{f}$.
i.e. 
\[
P(f)=\log\beta_{f},
\]
 where $\beta_{f}$ is a simple maximal positive eigenvalue of $\mathcal{L}_{f}$
given in Theorem \ref{thm: RPF} (the Ruelle-Perron-Frobenious theorem).
\end{defn}

\begin{rem}
One can prove that we can characterize the pressure via periodic points
of $(\Sigma_{A}^{+},\sigma)$. More precisely, for every $f\in C^{\alpha}(\Sigma_{A}^{+})$,
\[
P(f)=\lim_{n\to\infty}\frac{1}{n}\log\left(\sum_{\sigma^{n}\underline{x}=\underline{x}}e^{f^{n}(\underline{x})}\right),
\]
where $f^{n}(\underline{x})=f(\underline{x})+f(\sigma\underline{x})+...+f(\sigma^{n-1}\underline{x}).$ 
\end{rem}

\begin{defn}
Let $f$ be a Hölder continuous function, and suppose $P(f)=0$. We
define the\textit{ equilibrium state} of $f$ to be the measure 
\[
m_{f}=\frac{v_{f}\cdot\mu_{f}}{\int v_{f}\mathrm{d}\mu_{f}},
\]
 where $\mu_{f}$ is the unique probablity measure associated with
$f$ given in Theorem \ref{thm: RPF} (the R-P-F theorem).
\end{defn}

\textsf{We notice that the equilibrium state is an ergodic, $\sigma$-invarant
probability measure with positive entropy, and it is unique up to
Livšic cohomology for Hölder continuous functions.}

\begin{defn}
The \textit{variance} of $w\in C(\Sigma_{A}^{+})$ with respect to
a $\sigma-$invariant probablity measure $\mu$ is defined by 

\[
\mathrm{Var}(w,\mu)=\lim_{n\to\infty}\frac{1}{n}\int\left(w^{n}(\underline{x})-n\int w\mathrm{d}\mu\right)^{2}\mbox{d}\mu(\underline{x}),
\]
where $w^{n}(\underline{x})={\displaystyle \sum_{i=0}^{n-1}w\circ\sigma^{i}(\underline{x}).}$
\end{defn}

\textsf{In this note, we only consider variances with respect to equilibrium
states, which simplifies the discussion. In the sequel, we list first
recall some facts and useful formulas in calculating variances with
respect to equilibrium states. }

\begin{thm}
[Lemma 3.4 \cite{Jordan:2007dk}] \label{thm:(Jordan-Pollicott-)}

Suppose $f\in C^{\alpha}(\Sigma_{A}^{+})$, $\mathcal{L}_{f}1=1$
and $\int w\mbox{\ensuremath{\mathrm{d}}}m_{f}=0$, then 
\[
\mathrm{Var}(w,m_{f})=\int_{\Sigma_{A}^{+}}w^{2}\mathrm{d}m_{f},
\]
 where $m_{f}$ is the equilibrium state with respect to $f$. 
\end{thm}

\textsf{We also recall serval useful properties of the pressure. }

\begin{thm}
[Analyticity of the Pressure,  Prop. 4.7, 4.12 \cite{Parry:1990tn}]\label{Thm: analyticity of pressure}

If $f,g\in C^{\alpha}(\Sigma_{A}^{+})$ and if $m_{f}$ is the equilibrium
state of $f$, then \begin{enumerate}

\item The function $t\to P(f+tg)$ is analytic.

\item $\mathrm{Var}(g,m_{f})=0$ if and only if $g$ is Livšic cohomologous
to a constant.

\end{enumerate}
\end{thm}

\textsf{By the the analyticity of the pressure, we are able to differentiate
the pressure, and more importantly derivatives of the pressure give
us some handy formulas. }

\begin{thm}
[Derivatives of the pressure, Prop. 4.10, 4.11 \cite{Parry:1990tn}, Theorem 2.2 \cite{McMullen:2008eh}]\label{thm:(Derivatives-of-the}

Let $\psi_{t}$ be a smooth path in $C^{\alpha}(\Sigma_{A}^{+})$,
$m_{0}=m_{\psi_{0}}$ be the equilibrium state of $\psi_{0}$, and
$\dot{\psi}_{0}:=\left.\frac{d\psi_{t}}{dt}\right|_{t=0}$. We then
have 
\begin{equation}
\left.\dfrac{dP(\psi_{t})}{dt}\right|_{t=0}=\int_{\Sigma_{A}^{+}}\dot{\psi}_{0}\mathrm{d}m_{0}\label{eq:1}
\end{equation}
 and, if the first derivative is zero, i.e. $\int_{\Sigma_{A}^{+}}\dot{\psi}_{0}m_{0}=0$,
then 
\begin{equation}
\left.\dfrac{d^{2}P(\psi_{t})}{dt^{2}}\right|_{t=0}=\mathrm{Var}(\dot{\psi}_{0},m_{0})+\int_{\Sigma_{A}^{+}}\ddot{\psi}_{0}\mathrm{d}m_{0}.\label{eq:2}
\end{equation}

\end{thm}

\textsf{We now state an important property found by Bowen, so-called
Bowen's formula, which relates the growth rate of weighted periodic
orbits (or topological entropy of a certain system) with the pressure.}\textcolor{red}{{} }

\begin{thm}
[Bowen's formula, Prop. 6.1 \cite{Parry:1990tn}] \label{thm: Bowen's formula}If
$f:\Sigma_{A}^{+}\to\real$ is a positive continuous function, then
\[
P(-s\cdot f)=0
\]
if and only if $s=h_{f}$ where 
\[
h_{f}=\lim_{T\to\infty}\frac{1}{T}\log\#\{\underline{x}\in\Sigma_{A}^{+};\sigma^{n}\underline{x}=\underline{x}\mbox{ and }f^{n}(\underline{x})<T\mbox{ for some }n\in\mathbb{N}\}.
\]
 
\end{thm}

\textsf{Now we change gear and focus on a particular type of Hölder
continuous functions: functions only depending on first two coordinates,
i.e. $f:\Sigma_{A}^{+}\to\real$ and $f(\underline{x})=f(x_{0},x_{1})$,
where $\underline{x}=x_{0}x_{1}x_{2}...$ For every such function
$f$, we have explicit formulas of the eigenvalue $\beta_{f}$, the
eigenfunction $v_{f}$, and the measure $\mu_{f}$ given in Theorem}
\ref{thm: RPF}.

\begin{prop}
[Remark 1, p.27 \cite{Parry:1990tn}]\label{prop: locally const fucn}Let
$f:\Sigma_{A}^{+}\to\real$ be a function depending only on first
two coordinates and $\mathcal{L}_{f}v_{f}=\beta_{f}v_{f}$. Suppose
$A_{f}$ is the matrix defined by 
\[
A_{f}(i,j)=A(i,j)e^{f(i,j)},\mbox{ }1\leq i,j\leq k,
\]
 then \begin{enumerate}

\item $\beta_{f}$ is a maximal eigenvalue of $A_{f}$. 

\item $v_{f}$ is the (left) eigenvector $\mathbf{v_{f}}$ of $A_{f}$
w.r.t. $\beta_{f}$, i.e. 
\[
\sum_{i}\mathbf{v}_{f}(i)A(i,j)e^{f(i,j)}=\beta_{f}\cdot\mathbf{v}_{f}(j).
\]

\item If we define $g(i,j)=\log\mathbf{v}_{f}(i)-\log\mathbf{v}_{f}(j)-\log\beta_{f}+f(i,j)$,
then $\mathcal{L}_{g}1=1$ and the matrix $P_{f}$ corresponding to
$\mathcal{L}_{g}$ is 
\[
P_{f}(i,j)=A(i,j)e^{g(i,j)}=\frac{A(i,j)\mathbf{v}_{f}(i)e^{f(i,j)}}{\beta_{f}\cdot\mathbf{v}_{f}(j)}.
\]
Moreover, $P_{f}$ is column stochastic, i.e. $\sum_{i}A(i,j)e^{g(i,j)}=1$,
and the equilibrium state $m_{g}$ w.r.t. $g$ is given by 
\[
m_{g}[i_{0},..,i_{n}]=P_{f}(i_{0},i_{1})\cdot...\cdot P_{f}(i_{n-1},i_{n})\mathbf{p}_{f}(i_{n}),
\]
where $P_{f}\mathbf{p}_{f}=\mathbf{p}_{f}$ and $\sum_{i}\mathbf{p}_{f}(i)=1$
and we use the notation $[i_{0},i_{1},..,i_{n}]=\left\{ \underline{x}\in\Sigma_{A}^{+};x_{j}=i_{j},j=0,1,...,n\right\} .$

\end{enumerate}
\end{prop}

\begin{rem}
We call $[i_{0},i_{1},...,i_{n}]$ a cylinder set, and in fact cylinder
sets form a basis of the topology on $\Sigma_{A}^{+}$.
\end{rem}

\subsubsection{Pressure type metric on subshifts of finite type\label{sub:Pressure-metric}}

\textsf{Here we keep the same setting as in the previous subsection
that $(\Sigma_{A}^{+},\sigma)$ is a subshift of finite type and $A$
is irreducible. We consider the space $\mathcal{P}(\Sigma_{A}^{+})$
of Livšic cohomology classes of pressure zero Hölder continuous on
$\Sigma_{A}^{+}$, i.e. 
\[
\mathcal{P}(\Sigma_{A}^{+}):=\left\{ f;f\in C^{\alpha}(\Sigma_{A}^{+})\mbox{ for some \ensuremath{\alpha\:}and \ensuremath{P(f)=0}}\right\} /\sim,
\]
where $\sim$ denotes the Livšic cohomology relation.}

\textsf{The tangent space of $\mathcal{P}(\Sigma_{A}^{+})$ at $f$
is defined by 
\[
T_{f}\mathcal{P}(\Sigma_{A}^{+})=\ker\mathbf{\mathrm{\mathbf{D}}}P(f)=\left\{ \phi;\mbox{ \ensuremath{\phi\in C^{\alpha}(\Sigma_{A}^{+})\mbox{ for some \ensuremath{\alpha\:}and}\int\phi\mathrm{d}m_{f}=0}}\right\} /\sim,
\]
where $\mathbf{\mathrm{\mathbf{D}}}P(f)$ is derivative of $P$ at
$f$ and the $m_{f}$ is the equilibrium state of $f$.}

\textsf{Since the variance vanishes only on functions that are cohomologous
to zero, by using the variance we can define two }\textit{pressure
type metrics}\textsf{ on $T_{f}\mathcal{P}(\Sigma_{A}^{+})$ as what
stated below. The main reason why we call these metrics ``pressure
type metrics'' is that these metrics are defined via the variance
which is indeed the second derivative of the pressure (cf. Theorem
\ref{thm:(Derivatives-of-the}). }

\begin{defn}
[Pressure type metrics for subshifts of finite type] Let $f\in\mathcal{P}(\Sigma_{A}^{+})$
and $\phi\in T_{f}\mathcal{P}(\Sigma_{A}^{+})$ then we define two
\textit{pressure type metrics}: 
\[
\left\Vert \phi\right\Vert _{P}^{2}:=\mathrm{Var}(\phi,m_{f})
\]
and 
\[
\left\Vert \phi\right\Vert _{WP}^{2}:=\frac{\mathrm{Var}(\phi,m_{f})}{-\int f\mbox{d}m_{f}}.
\]
The former one is called the \textit{pressure metric }on $\mathcal{P}(\Sigma_{A}^{+})$
and denoted by $||\cdot||_{P}$, and the latter one is called the\textsl{
}\textit{Weil-Petersson metric} on $\mathcal{P}(\Sigma_{A}^{+})$
and denoted by $||\cdot||_{WP}$.
\end{defn}

\textsf{One can immediately see from the definition that $||\cdot||_{P}$
and $||\cdot||_{WP}$ are conformal and they only differ by a normalization.
We call $||\cdot||_{WP}$ the Weil-Petersson metric, because McMullen
points out that after normalizing by the variance, one can get the
Weil-Petersson metric of Teichmüller space (cf. the Theorem \ref{thm: pressue =00003D wp}). }

\textsf{The following theorem is our main formula for computing the
pressure metric and the Weil-Petersson metric. }

\begin{prop}
\label{prop: pressure metric}If $\{\phi_{t}\}_{t\in(-1,1)}$ is a
smooth one parameter family contained in $\mathcal{P}(\Sigma_{A}^{+})$,
then 
\[
\left\Vert \dot{\phi}_{0}\right\Vert _{WP}^{2}=\frac{\int\ddot{\phi}_{0}\mbox{d}m_{\phi_{0}}}{\int\phi_{0}\mbox{d}m_{_{0}}}\mbox{ and }\left\Vert \dot{\phi}_{0}\right\Vert _{P}^{2}=-\int\ddot{\phi}_{0}\mbox{d}m_{\phi_{0}}
\]
 where $\dot{\phi}_{0}=\left.\frac{d}{dt}\phi_{t}\right|_{t=0}$ and
$\ddot{\phi}_{0}=\left.\frac{d^{2}}{dt^{2}}\phi_{t}\right|_{t=0}$.\end{prop}
\begin{proof}
This follows the direct computation of the (Gâteaux) second derivative
of $P(\phi_{t})$: 
\begin{align*}
\left.\frac{d^{2}}{dt^{2}}P(\phi_{t})\right|_{t=0} & =(\mathbf{D}^{2}P)(\phi_{0})(\dot{\phi}_{0},\dot{\phi}_{0})+(\mathbf{D}P)(\phi_{0})(\ddot{\phi}_{0})\\
 & =\mathrm{Var}(\dot{\phi}_{0},m_{\phi_{0}})+\int\ddot{\phi}_{0}\mbox{d}m_{\phi_{0}}.
\end{align*}

Since $P(\phi_{t})=0$, we have 
\[
\left\Vert \dot{\phi}_{0}\right\Vert _{WP}^{2}:=\frac{\mathrm{Var}(\dot{\phi}_{0},m_{\phi_{0}})}{-\int\phi_{0}\mbox{d}m_{c_{0}}}=\frac{\int\ddot{\phi}_{0}\mbox{d}m_{\phi_{0}}}{\int\phi_{0}\mbox{d}m_{\phi_{0}}}.
\]

\end{proof}

\subsection{The Weil-Petersson metric on moduli spaces\label{sub:The-Weil-Petersson-metric-moduli }}

\textsf{In this subsection, we recall several facts of the Weil-Petersson
metric on moduli spaces. These geometric features of the Weil-Petersson
metric are critical clues for us to investigate the pressure metric
geometry. The following approach to the Weil-Petersson metric is not
traditional. The classic construction of the Weil-Petersson metric
is built on the complex structure of Teichmüller space. Whereas, without
employing the complex structure, McMullen in \cite{McMullen:2008eh}
gave a new characterization of the Weil-Petersson metric via the thermodynamic
formalism. Our study in this work is inspired by this point of view.
We summarize McMullen's theorem in Theorem \ref{thm: pressue =00003D wp}.
Before we jump into the statement, we recall some definitions. }

\textsf{Given a compact topological surface $S$ with negative Euler
characteristic, Teichmüller space $\tei(S)$ describes the marked
Riemannian metrics on $S$. i.e. the conformal classes of Riemannian
metric on $S$ with constant Gaussian curvature $-1$. The moduli
space $\moduli(S)$ describes the unmarked Riemannian metrics on $S$
and is obtained by quotienting $\tei(S)$ by the Mapping Class Group
of $S$. }

\textsf{We consider a $C^{1}$ family of metric $g_{\lambda}\in\moduli(S)$,
$0\leq\lambda\leq1$. Let $T^{1}S$ be the unit tangent bundle of
the surface $S$ with respect to the metric $g_{\lambda_{0}}$. Let
$\mu_{\lambda_{0}}$ be the corresponding Liouville measure on $T^{1}S$.
We denote by $\phi_{t}^{(\lambda_{0})}:T^{1}S\to T^{1}S$ the geodesic
flow. Since $g_{\lambda}$, for $0\leq\lambda\leq1$, is a volume
preserving deformation we have 
\[
\int\dot{g}_{\lambda_{0}}(v,v)\mathrm{d}\mu_{\lambda_{0}}(v)=0,
\]
where $\dot{g}_{\lambda_{0}}$ is defined via the expansion 
\[
g_{\lambda}=g_{\lambda_{0}}+\dot{g}_{\lambda_{0}}(\lambda-\lambda_{0})+O((\lambda-\lambda_{0})^{2}).
\]
(cf. Lemma 7. (a) and (c) \cite{Pollicott:1994wq}.)}

\begin{defn}
The \textit{variance} for $\dot{g}_{\lambda_{0}}(v,v)$ is given by
\[
\mathrm{Var}(\dot{g}_{\lambda_{0}},\mu_{\lambda_{0}}):=\lim_{t\to\infty}\frac{1}{t}\int\left(\int_{0}^{t}\dot{g}_{\lambda_{0}}(\phi_{s}^{(\lambda_{0})}(v),\phi_{s}^{(\lambda_{0})}(v))\mbox{d}s\right)^{2}\mbox{d}\mu_{\lambda_{0}}
\]
 
\end{defn}

\begin{thm}
[McMullen, Theorem 1.12 \cite{McMullen:2008eh}]\label{thm: pressue =00003D wp}

The Weil-Petersson metric is proportional to the variance. More precisely,
\[
\left\Vert \dot{g}_{\lambda_{0}}\right\Vert _{Pressure}^{2}:=\frac{\mathrm{Var}(\dot{g}_{\lambda_{0}},\mu_{\lambda_{0}})}{\int_{T^{1}S}g_{\lambda_{0}}(v,v)\mathrm{d}\mu_{\lambda_{0}}}=\frac{4}{3}\cdot\frac{\left\Vert \dot{g}_{\lambda_{0}}\right\Vert _{WP}^{2}}{\mathrm{area}(S,g_{\lambda_{0}})}.
\]

\end{thm}

\section{Pressure metrics on the space of metric graphs\label{sec:Pressure-metrics-on-graphs}}

\textsf{In what follows, $\mathcal{G}$ denotes a finite, connected,
nontrivial (i.e. which contain at least two distinct closed path)
and undirected graph with edge set $\mathcal{E}$. The length of each
edge is given by the edge weighting function} $l:\mathcal{E}\to\real_{>0}$
.

\begin{defn}
Let $\mathcal{M_{G}}$ denote the space of all edge weightings $l:\mathcal{E}\to\real_{>0}$
on $\gr$.
\end{defn}

\begin{defn}
The entropy $h(l)$ of the metric graph $(\mathcal{G},l)$ is defined
by 
\[
h(l)=\lim_{T\to\infty}\frac{1}{T}\log\#\{\gamma;l(\gamma)<T\},
\]
where $\gamma=(e_{0},e_{1},...,e_{n}=e_{0})$ is a closed cycle of
edges in $\mathcal{G}$ (without backtracking) and $l(\gamma)={\displaystyle \sum_{i=0}^{n-1}l(e_{i})}$.
\end{defn}

\begin{defn}
$\mathcal{M}_{\mathcal{G}}^{1}=\left\{ l:\mathcal{E}\to\real_{>0};h(l)=1\right\} $
is the space of all edge weightings with entropy $h(l)=1$. 
\end{defn}

\begin{rem}
\label{Rmk: why h=00003D1 is enough}One might consider taking other
normalizations on the space $\mathcal{M_{G}}$, for example normalize
$\mathcal{M_{G}}$ by the volume of the graph, i.e. consider the moduli
space $\mathcal{M}_{\mathcal{G}}^{v}:=\left\{ l:\mathcal{E}\to\real_{>0};\mathrm{Vol}(\mathcal{G})=\sum l(e_{i})=1\right\} $.
Notice that in the construction of pressure type metrics for subshifts
of finite type we only consider functions of pressure zero . In general,
for any positive Hölder functions $f$, there exists a positive number
$h_{f}$ such that $-h_{f}f$ is pressure zero (Bowen's formula cf.
Theorem\ref{thm: Bowen's formula}). Therefore, for each edge weighting
function $l,$ we can always scale $l$ to be the pressure zero function
$-h_{l}l$. So, for example the pressure type metrics on $\mathcal{M}_{\mathcal{G}}^{v}$
are the pullback the pressure type metrics on $\mathcal{M}_{\mathcal{G}}^{1}$
by the map $S:\mathcal{M}_{\mathcal{G}}^{v}\to\mathcal{M}_{\mathcal{G}}^{1}$
where $S(l)=-h_{l}l$. Hence, it is enough to study pressure type
metrics on $\mathcal{M}_{\mathcal{G}}^{1}$.
\end{rem}

\subsection{From undirected graphs to directed graphs\label{sub:From-undirected-graphs}}

\textsf{In this subsection, we closely follow the construction of
a symbolic model associating with metric graphs given in Pollicott
and Sharp's work }\cite{Pollicott:2014fs}. \textsf{It is well-known
that we can associate each directed graph with an adjacency matrix
which records the directed edges connecting vertices to vertices. }

\textsf{To put directions on the undirected graph, we do the following.
Given a undirected graph $\mathcal{G}$, for each edge $e\in\mathcal{E}$
we associate $e$ with two directed edges which, abusing notation,
we shall denote by $e$ and $\overline{e}$, i.e. two opposite directions
$\centerdot\stackrel[e]{\overline{e}}{\leftrightarrows}\centerdot$.
We denote by $\mathcal{E}^{0}$ the set of all directed edges. We
say $e'\in\mathcal{E}^{o}$ follows $e\in\mathcal{E}^{0}$ if $e'$
begins at the terminal endpoint of $e$, i.e. $\underset{e'}{\nwarrow}\centerdot\underset{e}{\swarrow}$.
We then define a $|\mathcal{E}^{o}|\times|\mathcal{E}^{o}|$ matrix
$A$, with rows and columns indexed by $\mathcal{E}^{0}$, by 
\[
A(e,e')=\begin{cases}
1 & \mbox{if }e'\mbox{ follows }e\mbox{ and }e'\neq\overline{e}\\
0 & \mbox{otherwise.}
\end{cases}
\]
}

\textsf{Then the shift space 
\[
\Sigma_{A}=\left\{ \underline{e}=(e_{n})_{n\in\mathbb{Z}}\in(\mathcal{E}^{o})^{\mathbb{Z}};A(e_{n},e_{n+1})=1\mbox{ }\forall n\in\mathbb{Z}\right\} 
\]
can be naturally identified with the space of all two-sided infinite
path (with a distinguished zeroth edge) in the graph $\mathcal{G}$.
We call $A$ the }\textit{adjacency matrix }\textsf{of the undirected
graph $\mathcal{G}$.}
\begin{rem}
It is not hard to see $A$ is irreducible. Recall that the graph $\mathcal{G}$
is connected, so after we associate two (opposite) directions to each
edge of $\mathcal{G}$, we know that for each pair of vertices of
$\mathcal{G}$ there exists a directed path connecting them. 
\end{rem}
\textsf{Besides, Bowen showed that each two-sided subshift of finite
type $(\Sigma_{A},\sigma)$ can be characterized by an one-sided subshift
of finite type $(\Sigma_{A}^{+},\sigma)$. More precisely, each Hölder
function on $\Sigma_{A}$ (two-sided sequences) is cohomologous to
a Hölder function only depending on the future coordinates (one-sided
sequences). Thus, it is enough to study the one-sided subshift of
finite type $(\Sigma_{A}^{+},\sigma)$, which is what we will do in
the following.}

\textsf{The edge weighting function $l:\mathcal{E}\to\real_{>0}$,
our main player, is clearly well-defined under the directed graph
setting. Moreover, each edge weighting could be identified as a locally
constant function on $\Sigma_{A}^{+}$ by, abusing the notation, $l(\underline{e}):=l(e_{0})$
where $\underline{e}=e_{0}e_{1}....\in\Sigma_{A}^{+}$. Therefore,
$\mathcal{M_{G}}$ can be naturally regarded as a set of positive
locally constant functions on $\Sigma_{A}^{+}$. By definition, we
can rewrite the entropy of the metric graph $(\mathcal{G},l)$ by
\begin{align*}
h(l):= & \lim_{T\to\infty}\frac{1}{T}\log\#\left\{ \gamma;l(\gamma)<T\right\} \\
= & \lim_{T\to\infty}\frac{1}{T}\log\#\left\{ \underline{e}\in\Sigma_{A}^{+};\sigma^{n}(\underline{e})=\underline{e},\mbox{ and}{\displaystyle \sum_{i=0}^{n-1}l(e_{i})}<T\mbox{ for some }n\in\mathbb{N}\right\} \\
= & h_{l}.
\end{align*}
Hence, we have $\mathcal{M}_{\mathcal{G}}^{1}=\left\{ l\in C(\Sigma_{A}^{+});\mbox{ }l>0\mbox{, }l(\underline{e})=l(e_{0})\mbox{ and \ensuremath{h_{l}=1}}\right\} .$}

\begin{lem}
[Lemma 3.1 \cite{Pollicott:2014fs}].\begin{enumerate}

\item The pressure function is analytic on the space of locally continuous
functions.

\item The entropy $h(l)$ of $(\mathcal{G},l)$ is characterized
by $P(-h(l)\cdot l)=0$.

\item The entropy function $\mathcal{M_{G}}\ni l\mapsto h(l)\in\real_{>0}$
varies analytically for $l>0.$ 

\end{enumerate}
\end{lem}

\begin{proof}
The first two assertions are coming from Theorem\ref{Thm: analyticity of pressure}
and Theorem \ref{thm: Bowen's formula}, respectively. The last one
is a consequence of the implicit function theorem (for Banach spaces). 
\end{proof}

\subsection{Two pressure type metrics\label{sub: pressure type on gh and formula}}

\textsf{Following the discussion in Section \ref{sub:Pressure-metric},
we can similarly define the pressure type metrics on $\mathcal{M}_{\mathcal{G}}^{1}$.
Specifically, because} $l\in\mathcal{M}_{\mathcal{G}}^{1}$\textsf{,
we know $h(l)=1$ and $P(-h(l)l)=P(-l)=0$. Hence for each $l\in\mathcal{M}_{\mathcal{G}}^{1}$
there exists an equilibrium state $m_{-l}$ with respect to the function
$-l$.}

\begin{defn}
We define the tangent space to $\mathcal{M_{G}}^{1}$ at $l$ by 
\[
T_{l}\mathcal{M}_{\mathcal{G}}^{1}=\left\{ \phi:\Sigma_{A}^{+}\to\real;\mbox{ \ensuremath{\phi(\underline{e})=\phi(e_{0})}}\mbox{ for all \ensuremath{e\in\mathcal{E}^{0}}, and}\int\phi\mathrm{d}m_{-l}=0\right\} .
\]

\end{defn}

\begin{defn}
For $\phi\in T_{l}\mathcal{M}_{\mathcal{G}}^{1}$ then we define two
pressure type metrics: 
\[
\left\Vert \phi\right\Vert _{P}^{2}:=\mathrm{Var}(\phi,m_{-l})
\]
and 
\[
\left\Vert \phi\right\Vert _{WP}^{2}:=\frac{\mathrm{Var}(\phi,m_{-l})}{\int l\mbox{d}m_{-l}}.
\]
The former one is called the \textit{pressure metric} on $\mathcal{M}_{\mathcal{G}}^{1}$
and denoted by $||\cdot||_{P}$, and the latter one is called the
\textit{Weil-Petersson metric} on $\mathcal{M}_{\mathcal{G}}^{1}$
and denoted by $||\cdot||_{WP}$. We can then define the length of
every continuously differentiable curve $\gamma:[0,1]\to\mathcal{M_{G}}^{1}$
by
\[
L_{P}(\gamma)=\int_{0}^{1}||\dot{\gamma}||_{P}\mathrm{d}t\mbox{ and }L_{WP}(\gamma)=\int_{0}^{1}||\dot{\gamma}||_{WP}\mathrm{d}t
\]
 and thus define two path space metrics on $\mathcal{M}_{\mathcal{G}}^{1}$
by $d_{WP}(l_{1},l_{2})=\inf_{\gamma}\{L_{WP}(\gamma)\}$ and $d_{P}(l_{1},l_{2})=\inf_{\gamma}\{L_{P}(\gamma)\}$,
where the infimum is taken over all continuously differentiable curves
with $\gamma(0)=l_{1}$ and $\gamma(1)=l_{2}$. 
\end{defn}

\begin{rem}
\begin{enumerate}-

\item One should notice that our Weil-Petersson metric for graphs
is different from the Weil-Petersson metric for graphs mentioned in
Pollicott and Sharp's paper \cite{Pollicott:2014fs}. In fact, their
Weil-Petersson type metric is our pressure metric. Readers should
be careful about this difference. 

\item Another reason that we consider the new pressure type, $||\cdot||_{WP}$,
is because it also coincides with the definition of the pressure metric
on the Hitchin component introduced by Bridgeman, Canary, Labourie
and Sambarino in \cite{Bridgeman:2013to}.

\end{enumerate}
\end{rem}

\textsf{Because $l$ is a locally constant function, most of the above
formulas of pressure type metrics, variances, and equilibrium states
could be simplified quite a bit. The remaining of this subsection
dedicates to expressing those quantities in simpler forms.}

\textsf{By Proposition \ref{prop: locally const fucn}, we consider
the matrix $A_{-l}$ associating to $-l$ defined by $A_{-l}(i,j)=A(i,j)e^{-l(i)}$.
Notice that because $h(l)=1$, $A_{-l}$ has a simple maximum eigenvalue
$1$. Let $\mathbf{v}_{-l}$ be the left eigenvector of $A_{-l}$
w.r.t. $1$, i.e. $\mathbf{v}_{-l}\cdot A_{-l}=\mathbf{v}_{-l}$.
Then we construct the column stochastic matrix $P_{-l}$ associating
with $-l$: 
\begin{equation}
P_{-l}(i,j)=\frac{A(i,j)\mathbf{v}_{-l}(i)e^{f(i)}}{\mathbf{v}_{-l}(j)},\label{eq: normalized weight matrix}
\end{equation}
and let $\mathbf{p}_{-l}$ be the unit right eigenvector of $P_{-l}$
w.r.t. $1$, i.e. $P_{-l}\mathbf{p}_{-l}=\mathbf{p}_{-l}$ and $\sum_{i}\mathbf{p}_{-l}(i)=1$.
Thus we know that the equilibrium state $m_{-l}$ is given by 
\[
m_{-l}[i_{0},i_{1},...,i_{n}]=P_{-l}(i_{0},i_{1})\cdot...\cdot P_{-l}(i_{n-1},i_{n})\mathbf{p}_{-l}(i_{n}).
\]
When there is no ambiguity, for short, we will drop the subscript
$l$ from $\mathbf{v}_{-l}$,$P_{-l}$ and $\mathbf{p}_{-l}$ and
denote them by $\mathbf{v}$, $P$ and $\mathbf{p}$, respectively. }

\begin{prop}
\label{Prop: var formula graph}We have \begin{enumerate}

\item For each $l\in\mathcal{M}_{\mathcal{G}}^{1}$, the tangent
space to $\mathcal{M}_{\mathcal{G}}^{1}$ at $l$ is 
\[
T_{l}\mathcal{M_{G}}^{1}=\left\{ \phi\in C(\Sigma_{A}^{+});\phi(\underline{e})=\phi(e_{0})\mbox{ and }{\displaystyle \sum_{e\in\mathcal{E}^{o}}\phi(e)\mathbf{p}_{e}}=0\right\} ,
\]
where $\mathbf{p}_{e}=\mathbf{p}(e)$.

\item Let $l_{t}$ be a smooth path in $\mathcal{M_{G}}^{1}$, and
$g_{t}$ be the (up to cohomology) normalized function cohomologous
to $-l_{t}$, i.e. $\mathcal{L}_{g_{t}}1=1$ and $-l_{t}=g_{t}+h_{t}\circ\sigma-h_{t}$,
and let $P(i,j)=P_{-l}(i,j)=A(i,j)e^{g_{0}(i,j)}$, and $\mathbf{p}$
be the right eigenvector of $P$ w.r.t. $1$ (i.e. $P\mathbf{p}=\mathbf{p}$)
then

\begin{align}
\mathrm{Var}(-\dot{l}_{0},m_{-l_{0}}) & =\int(\ddot{l}_{0})\mathrm{d}m_{-l_{0}}=\sum_{i}\ddot{l}_{0}(i)\mathbf{p}_{i}\label{eq: var twice derivateve}\\
 & =\int(\dot{g}_{0})^{2}\mathrm{d}m_{g_{0}}=\sum_{i,j}\left(\dot{g}_{0}(i,j)\right)^{2}P_{ij}\mathbf{p}_{j}.\label{eq: var squre}
\end{align}

\end{enumerate}
\end{prop}

\begin{proof}
(1): It is obvious. 

(2): By Proposition \ref{prop: locally const fucn}, we know that
$m_{-l_{0}}[e_{i}]=\mathbf{p}_{e_{i}}$ for each $i$ and $m_{-l_{0}}[e_{i},e_{j}]=P(e_{i},e_{j})\mathbf{p}(e_{j})$
for all $e_{i},e_{j}\in\mathcal{E}^{0}$.

For equation $($\ref{eq: var twice derivateve}$)$, by Theorem \ref{thm:(Derivatives-of-the},
we know 
\[
\mathrm{Var}(-\dot{l}_{0},m_{-l_{0}})=\int(\ddot{l}_{0})\mathrm{d}m_{-l_{0}}=\sum_{i}\int_{[i]}(\ddot{l}_{0})\mbox{d}m_{-l_{0}}=\sum_{i}\ddot{l}_{0}(i)\mathbf{p}_{i}.
\]

For equation $($\ref{eq: var squre}$)$, since $P(-h(l)l)=0$ and
by Theorem \ref{thm:(Derivatives-of-the}, we have $\int\dot{l}_{0}\mbox{d}m_{-l_{0}}=0$.
Moreover, because for each Hölder continuous function the equilibrium
state is unique in each Livšic cohomology class we know that $m_{-l_{0}}=m_{g_{0}}$.
Since $-\dot{l}_{0}\sim\dot{g}_{0}$ and $-\ddot{l}_{0}\sim\ddot{g}_{0}$,
we have $\int\dot{l}_{0}\mbox{d}m_{-l_{0}}=\int\dot{g}_{0}\mbox{d}m_{g_{0}}=0$
and $\int-\ddot{l}_{0}\mathrm{d}m_{-l_{0}}=\int\ddot{g}_{0}\mathrm{d}m_{g_{0}}.$
Because $\mathcal{L}_{g_{0}}1=1$ and $\int\dot{g}_{0}\mbox{d}m_{g_{0}}=0$,
by Theorem \ref{thm:(Jordan-Pollicott-)} we have 
\[
\mathrm{Var}(-\dot{l}_{0},m_{-l_{0}})=\mathrm{Var}(\ddot{g}_{0},m_{g_{0}})=\int(\dot{g}_{0})^{2}\mathrm{d}m_{g_{0}}=\sum_{i.j}\int_{[i,j]}(\dot{g}_{0})^{2}\mathrm{d}m_{g_{0}}=\sum_{i,j}\left(\dot{g}_{0}(i,j)\right)^{2}P_{ij}\mathbf{p}_{j}.
\]

\end{proof}

\begin{rem}
\label{rem: mistake in PS}Equation $($\ref{eq: var squre}$)$ is
very close to the formula given in Lemma 3.3 \cite{Pollicott:2014fs}.
However, Lemma 3.3 in \cite{Pollicott:2014fs} is NOT true for all
functions in the tangent space $T_{l}\mathcal{M}_{\mathcal{G}}^{1}$.
It is only true for normalized functions, i.e. $f\in T_{l}\mathcal{M}_{\mathcal{G}}^{1}$
and $\mathcal{L}_{f}1=1$. 
\end{rem}

\textsf{Notice that it is convenient to interpret $\mathcal{M}_{\mathcal{G}}^{1}$
as a hypersurface in $\mathcal{M_{G}}$. Because each edge weighting
function $l$ is depending on the weighting (i.e. length) of each
edge, it can be regarded as a $k-$variable function $l=l(e_{1},e_{2},...,e_{k})$
where $e_{1}$, $e_{2}$,...,$e_{k}$ are edges of $\mathcal{G}$.
Thus, in this perspective $\mathcal{M_{G}}$ is $\real_{>0}^{k}$.
Moreover, the condition $h(l)=1$ gives us an equation of} $e_{1},e_{2}$,...,
\textsf{and} $e_{k}$, \textsf{so $\mathcal{M}_{\mathcal{G}}^{1}$
is a co-dimension one submanifold in $\mathcal{M_{G}}.$}

\section{Examples\label{sec:Examples}}

\textsf{For examples in this section, we follow the recipe below to
calculate the pressure type metrics. Let $\mathcal{G}$ be a undirected
finite graph and $l$ be an edge weighting function then}

\begin{itemize}

\item\textsf{ First, associate two opposite directions to each edge
as we defined in Section \ref{sub:From-undirected-graphs}, then write
down the adjacency matrix $A$ and the wighted adjacency matrix $A_{-h(l)l}$
associated with $l$.}

\item \textsf{Second, solve the equation $h(l)=1$. Explicitly, since
$h(l)=1$ means that $1$ is the an eigenvalue of the matrix $A_{-l}$,
the characteristic polynomial of $A_{-l}$ is a function of edges
$e_{1}$,...,$e_{k}$ of $\mathcal{G}$, i.e. $\det(A_{-l}-\mathrm{Id})=0$.
By solving the characteristic polynomial, we can write $l=(l(e_{1}),l(e_{2}),...,l(e_{k}))$
where $l(e_{k})=w(l(e_{1}),...,l(e_{k-1}))$ is an analytic function
depending $l(e_{1}),...,l(e_{k-1})$.}

\item \textsf{Third, compute the right eigenvector $\mathbf{v}_{-l}$
of $A_{-l}$ w.r.t the eigenvalue $1$.}

\item \textsf{Fourth, write down the normalized weighted matrix $P_{-l}$
in $($\ref{eq: normalized weight matrix}$)$, and compute the equilibrium
state $\mathbf{p}$ of $-l$, i.e. the unit left eigenvector $\mathbf{p}$
of $P_{-l}$ w.r.t. $1$.}

\item \textsf{Fifth, consider $l=(l(e_{1}),l(e_{2}),...,w(l(e_{1}),l(e_{2}),...,l(e_{k})))$
as a parametrization of $\mathcal{M}_{\mathcal{G}}^{1}$, and then
compute the tangent vectors $\frac{\partial}{\partial l(e_{i})}$
for $1\leq i\leq k.$}

\item \textsf{Last, compute the pressure metric and the Weil-Petersson
metric of $\frac{\partial}{\partial l(e_{i})}$ for $1\leq i\leq k.$}

\end{itemize}

\textsf{We notice that graphs of a belt buckle $\mathcal{G}_{2}$,
a dumbbell $\mathcal{G}_{3}$ and a three-petal rose $\mathcal{G}_{3}$
are graphs with 3 edges. Therefore the moduli spaces $\mathcal{M}_{\mathcal{G}_{2}}^{1}$,
$\mathcal{M}_{\mathcal{G}_{3}}^{1}$ and $\mathcal{M}_{\mathcal{G}_{3}}^{1}$
are two-dimensional manifolds sitting in $\mathbb{R}^{3}$. Moreover,
in the Riemannian geometry setting, it is natural to study the important
geometric invariants--curvatures. In dimension two, it is the Gaussian
curvature.}

\textsf{Before we start working on examples, we recall two useful
asymptotic notations that $A(x)\sim B(x)$ means ${\displaystyle \lim\frac{A(x)}{B(x)}=1}$
and $A(x)\asymp B(x)$ means there exist positive constants $c$ and
$c'$ such that $c<\frac{A(x)}{B(x)}<c'$ . }

\subsection{Figure 8 graphs}

\textsf{The first example is a figure 8 graph which we denote by $\mathcal{G}_{1}$.
The picture below is a figure 8. In the sake of brevity, we denote
$l(e_{1})$ and $l(e_{2})$ by $x$ and $y$, respectively. }

\begin{center}\includegraphics[scale=0.8]{figure_8} \end{center}

\textsf{Following the recipe, we have }

\begin{itemize}

\item$A=\left(\begin{array}{cccc}
1 & 1 & 0 & 1\\
1 & 1 & 1 & 0\\
0 & 1 & 1 & 1\\
1 & 0 & 1 & 1
\end{array}\right)$ \textsf{and} $A_{-l}=\left(\begin{array}{cccc}
e^{-x} & e^{-x} & 0 & e^{-x}\\
e^{-y} & e^{-y} & e^{-y} & 0\\
0 & e^{-x} & e^{-x} & e^{-x}\\
e^{-y} & 0 & e^{-y} & e^{-y}
\end{array}\right).$

\item$h(l)=1$ $\implies$$\det(\mathrm{Id}-A_{-l})=0$$\implies e^{-y}=\frac{1-e^{-x}}{1+3e^{-x}}.$ 

\item\textsf{ }$\mathbf{v}=\left(\frac{2}{1+3e^{-x}},1,\frac{2}{1+3e^{-x}},1\right).$

\item $P=\left(\begin{array}{cccc}
e^{-x} & \frac{2e^{-x}}{1+3e^{-x}} & 0 & \frac{2e^{-x}}{1+3e^{-x}}\\
\frac{1-e^{-x}}{2} & \frac{1-e^{-x}}{1+3e^{-x}} & \frac{1-e^{-x}}{2} & 0\\
0 & \frac{2e^{-x}}{1+3e^{-x}} & e^{-x} & \frac{2e^{-x}}{1+3e^{-x}}\\
\frac{1-e^{-x}}{2} & 0 & \frac{1-e^{-x}}{2} & \frac{1-e^{-x}}{1+3e^{-x}}
\end{array}\right)$ \textsf{and }$\mathbf{p}=\left(\begin{array}{c}
\mathbf{p}_{1}(x)\\
\mathbf{p}_{2}(x)\\
\mathbf{p}_{3}(x)\\
\mathbf{p}_{4}(x)
\end{array}\right)=\left(\begin{array}{c}
\frac{2e^{x}}{6e^{x}+e^{2x}-3}\\
\frac{2e^{x}+e^{2x}-3}{2\left(6e^{x}+e^{2x}-3\right)}\\
\mathbf{p}_{1}(x)\\
\mathbf{p}_{2}(x)
\end{array}\right)$.

\item \textsf{Consider the path} $\mathcal{M}_{\mbox{\ensuremath{\mathcal{G}}}_{1}}^{1}\ni l=(l_{1},l_{2})=(x,-\log\frac{1-e^{-x}}{1+3e^{-x}})=:c(x)$. 

\item $\dot{c}(x)=(1,-\frac{4e^{x}}{\left(e^{x}-1\right)\left(e^{x}+3\right)})$\textsf{
and} $\ddot{c}(x)=(0,\frac{4e^{x}\left(e^{2x}+3\right)}{\left(e^{x}-1\right)^{2}\left(e^{x}+3\right)^{2}}).$

\end{itemize}

\begin{prop}
\label{prop: fig 8 press incomplete}The moduli space of figure $8$
graphs is \textbf{incomplete} under the pressure metric $||\cdot||_{P}$.
i.e., $(\mathcal{M}_{\mathcal{G}_{1}}^{1},||\cdot||_{P})$ is incomplete.
\end{prop}

\begin{proof}
By Proposition \ref{Prop: var formula graph}, we have 
\[
||\dot{c}(x)||_{P}^{2}=\mathrm{Var}(\dot{c}(x),m_{c(x)})=2\cdot\frac{4e^{x}\left(e^{2x}+3\right)}{\left(e^{x}-1\right)^{2}\left(e^{x}+3\right)^{2}}\cdot\mathbf{p_{2}}(x)=\frac{4e^{x}\left(e^{2x}+3\right)}{-24e^{x}+6e^{2x}+8e^{3x}+e^{4x}+9}.
\]
When $x$ is close to zero, we have the expansion 
\[
||\dot{c}(x)||_{P}^{2}=\frac{1}{x}-\frac{5}{4}+\frac{85x}{48}+o(x^{2}).
\]

This estimate shows that when $x\to0$ 
\[
\sqrt{\int||\dot{c}(x)||_{P}^{2}\mbox{d}m_{c(x)}}\asymp\frac{1}{\sqrt{x}}.
\]
Finally, since $\int_{0}^{1}\frac{1}{\sqrt{x}}dx$ is convergent we
see that the metric is incomplete. i.e. the curve arrives at $l_{1}=x=0$
in finite time with respect to this metric. 
\end{proof}

\begin{rem}
There is a hidden natural condition $l_{2}>0$ that we have to take
into account. This condition was missing in Section 6 \cite{Pollicott:2014fs}.
For $\mathcal{G}_{1}$ the condition is $1-e^{-x}>0$ and $\frac{1-e^{-x}}{1+3e^{-x}}<1$,
which is equivalent to $x>0$.
\end{rem}

\begin{prop}
\label{prop: fig 8 WP complete }The moduli space of figure $8$ graphs
is \textbf{complete} under the Weil-Petersson metric $||\cdot||_{WP}$.
i.e.,$(\mathcal{M}_{\mathcal{G}_{1}}^{1},||\cdot||_{WP})$ is complete.
\end{prop}

\begin{proof}
Continue with the computation in the previous proposition, we know
that 
\[
||\dot{c}(x)||_{WP}^{2}=\frac{\mathrm{Var}(\dot{c}(x),m_{c(x)})}{{\displaystyle 2x\mathbf{p_{1}}(x)+2c(x)\mathbf{p_{2}}(x)}}=-\frac{4e^{x}\left(e^{2x}+3\right)}{\left(e^{x}-1\right)\left(e^{x}+3\right)\left(\left(2e^{x}+e^{2x}-3\right)\log\left(\frac{e^{x}-1}{e^{x}+3}\right)-4e^{x}x\right)}.
\]
When $x$ is close to zero, we know that 
\[
||\dot{c}(x)||_{WP}^{2}\asymp\frac{-1}{x^{2}\log x},
\]
and when $x$ tends to infinity we have 
\[
||\dot{c}(x)||_{WP}^{2}\asymp\frac{1}{x}.
\]
 This shows that the curve arrives the boundary of $\mathcal{M}_{\mathcal{G}_{1}}^{1}$,
i.e. $x=0$ and $x=\infty$, in infinite time. Because $\mathcal{M}_{\mathcal{G}_{1}}^{1}=\{(x,c(x));x\in(0,\infty)\}$
is a one dimensional smooth manifold, we can conclude it is complete. 
\end{proof}

\subsection{Belt buckles}

\textsf{The second example is a graph with two vertices, connected
to each other by three edges which we denote by $\mathcal{G}_{2}$
and call it a belt buckle. The picture below is a picture of a belt
buckle. For brevity, we denote $l(e_{1})$, $l(e_{2})$ and $l(e_{3})$
by $x$, $y$ and $z$, respectively. }

\begin{center}\includegraphics[scale=0.8]{beltbuckle} \end{center}

\textsf{Following the recipe, we have }

\begin{itemize}

\item$A=\left(\begin{array}{cccccc}
0 & 0 & 0 & 0 & 1 & 1\\
0 & 0 & 0 & 1 & 0 & 1\\
0 & 0 & 0 & 1 & 1 & 0\\
0 & 1 & 1 & 0 & 0 & 0\\
1 & 0 & 1 & 0 & 0 & 0\\
1 & 1 & 0 & 0 & 0 & 0
\end{array}\right)$ $A_{-l}=\left(\begin{array}{cccccc}
0 & 0 & 0 & 0 & e^{-x} & e^{-x}\\
0 & 0 & 0 & e^{-y} & 0 & e^{-y}\\
0 & 0 & 0 & e^{-z} & e^{-z} & 0\\
0 & e^{-x} & e^{-x} & 0 & 0 & 0\\
e^{-y} & 0 & e^{-y} & 0 & 0 & 0\\
e^{-z} & e^{-z} & 0 & 0 & 0 & 0
\end{array}\right).$

\item$h(l)=1$$\implies$ $\det(A_{-l}-\mathrm{Id})=0$$\implies$$e^{-z}=\frac{1-e^{-x-y}}{2e^{-x-y}+e^{-x}+e^{-y}}$\textsf{. }

\item$\mathbf{v}=\left(\frac{e^{-y}+1}{2e^{-x-y}+e^{-x}+e^{-y}},\frac{e^{-x}+1}{2e^{-x-y}+e^{-x}+e^{-y}},1,\frac{e^{-y}+1}{2e^{-x-y}+e^{-x}+e^{-y}},\frac{e^{-x}+1}{2e^{-x-y}+e^{-x}+e^{-y}},1\right).$

\item $P=\left(\begin{array}{cccccc}
0 & 0 & 0 & 0 & \frac{e^{-x}(e^{-y}+1)}{e^{-x}+1} & \frac{e^{-x}(e^{-y}+1)}{2e^{-x-y}+e^{-x}+e^{-y}}\\
0 & 0 & 0 & \frac{(e^{-x}+1)e^{-y}}{e^{-y}+1} & 0 & \frac{(e^{-x}+1)b}{2e^{-x-y}+e^{-x}+e^{-y}}\\
0 & 0 & 0 & \frac{1-e^{-x-y}}{e^{-y}+1} & \frac{1-e^{-x-y}}{e^{-x}+1} & 0\\
0 & \frac{e^{-x}(e^{-y}+1)}{e^{-x}+1} & \frac{e^{-x}(e^{-y}+1)}{2e^{-x-y}+e^{-x}+e^{-y}} & 0 & 0 & 0\\
\frac{(e^{-x}+1)e^{-y}}{e^{-y}+1} & 0 & \frac{(e^{-x}+1)b}{2e^{-x-y}+e^{-x}+e^{-y}} & 0 & 0 & 0\\
\frac{1-e^{-x-y}}{e^{-y}+1} & \frac{1-e^{-x-y}}{e^{-x}+1} & 0 & 0 & 0 & 0
\end{array}\right)$.

\item $\mathbf{p}=\left(\begin{array}{c}
\mathbf{p}_{1}(x,y)\\
\mathbf{p}_{2}(x,y)\\
\mathbf{p}_{3}(x,y)\\
\mathbf{p}_{4}(x,y)\\
\mathbf{p}_{5}(x,y)\\
\mathbf{p}_{6}(x,y)
\end{array}\right)=\left(\begin{array}{c}
\frac{e^{x}\left(e^{y}+1\right)^{2}}{4\left(3e^{x+y}+e^{2x+y}+e^{x+2y}-1\right)}\\
\frac{\left(e^{x}+1\right)^{2}e^{y}}{4\left(3e^{x+y}+e^{2x+y}+e^{x+2y}-1\right)}\\
\frac{\left(e^{x}+e^{y}+2\right)\left(e^{x+y}-1\right)}{4\left(3e^{x+y}+e^{2x+y}+e^{x+2y}-1\right)}\\
\mathbf{p}_{1}(x,y)\\
\mathbf{p}_{2}(x,y)\\
\mathbf{p}_{3}(x,y)
\end{array}\right)$ .

\item\textsf{ Now we consider the surface }$\mathcal{M}_{\mbox{\ensuremath{\mathcal{G}}}_{2}}^{1}\ni l=(l_{1},l_{2},l_{3})=(x,y,-\log\left(\frac{e^{x+y}-1}{e^{x}+e^{y}+2}\right))=(x,y,S(x,y))=l(x,y)$. 

\item $\frac{\partial}{\partial x}l=(1,0,-\frac{e^{x}\left(e^{y}+1\right)^{2}}{\left(e^{x}+e^{y}+2\right)\left(e^{x+y}-1\right)})$,
$\frac{\partial l}{\partial y}=(0,1,-\frac{\left(e^{x}+1\right)^{2}e^{y}}{\left(e^{x}+e^{y}+2\right)\left(e^{x+y}-1\right)}),$

$\frac{\partial^{2}}{\partial x^{2}}l=(0,0,\frac{e^{x}\left(e^{y}+1\right)^{2}\left(e^{2x+y}+e^{y}+2\right)}{\left(e^{x}+e^{y}+2\right)^{2}\left(e^{x+y}-1\right)^{2}}),$$\frac{\partial^{2}}{\partial x\partial y}l=(0,0,-\frac{\left(e^{x}+1\right)\left(e^{y}+1\right)e^{x+y}\left(e^{x+y}-e^{x}-e^{y}-3\right)}{\left(e^{x}+e^{y}+2\right)^{2}\left(e^{x+y}-1\right)^{2}})$, 

\textsf{and} $\frac{\partial^{2}}{\partial y^{2}}l=(0,0,\frac{\left(e^{x}+1\right)^{2}e^{y}\left(e^{x+2y}+e^{x}+2\right)}{\left(e^{x}+e^{y}+2\right)^{2}\left(e^{x+y}-1\right)^{2}}).$

\end{itemize}

\begin{rem}
For this graph, the hidden natural condition $l_{3}>0$ is equivalent
to 
\[
e^{x+y}<3+e^{x}+e^{y}.
\]

\end{rem}

\textsf{In this example, we are interested in the curvature of $\mathcal{M}_{\mathcal{G}_{2}}^{1}$
with respect to different metrics $||\cdot||_{P}$ and $||\cdot||_{WP}$.
Since $\mathcal{M}_{\mathcal{G}_{2}}^{1}$ is a two-dimensional manifold
in $\mathbb{R}^{3}$, in order to calculate the Gaussian curvature,
we first need to derive the corresponding first fundamental forms
\[
ds^{2}=E(x,y)dx^{2}+F(x,y)dxdy+G(x,y)dy^{2}
\]
 where $x=l_{1}$ and $y=l_{2}$. }

\begin{lem}
[Brioschi formula] If a metric has local coordinates 
\[
ds^{2}=E(x,y)dx^{2}+F(x,y)dxdy+G(x,y)dy^{2},
\]
then the curvature is given by 
\[
K(x,y)=\frac{\det\left|\begin{array}{ccc}
-\frac{E_{yy}}{2}+F_{xy}-\frac{Gxx}{2} & \frac{E_{x}}{2} & F_{x}-\frac{E_{y}}{2}\\
F_{y}-\frac{G_{x}}{2} & E & F\\
\frac{G_{y}}{2} & F & G
\end{array}\right|-\det\left|\begin{array}{ccc}
0 & \frac{E_{y}}{2} & \frac{G_{x}}{2}\\
\frac{E_{v}}{2} & E & F\\
\frac{G_{x}}{2} & F & G
\end{array}\right|}{\left(EG-F^{2}\right)^{2}}.
\]

\end{lem}

By direction computations, we have the following propositions.

\begin{prop}
The first fundamental form of $(\mathcal{M}_{\mathcal{G}_{2}}^{1},||\cdot||_{P})$
is 
\begin{align*}
E_{P}(x,y) & =\mathrm{Var}(\frac{\partial l}{\partial x},m_{-l})=2\cdot\frac{\partial^{2}S}{\partial x^{2}}\cdot\mathbf{p}_{3}(x,y)\\
 & =\frac{e^{x}\left(e^{y}+1\right)^{2}\left(e^{2x+y}+e^{y}+2\right)}{2\left(e^{x}+e^{y}+2\right)\left(e^{x+y}-1\right)\left(3e^{x+y}+e^{2x+y}+e^{x+2y}-1\right)},\\
F_{P}(x,y) & =\frac{1}{2}\left(\mathrm{Var}(\frac{\partial l}{\partial x}+\frac{\partial l}{\partial y},m_{-l})-\mathrm{Var}(\frac{\partial l}{\partial x},m_{-l})-\mathrm{Var}(\frac{\partial l}{\partial y},m_{-l})\right)\\
 & =2\cdot\frac{\partial^{2}S}{\partial x\partial y}\cdot\mathbf{p}_{3}(x,y).\\
 & =\frac{\left(e^{x}+1\right)\left(e^{y}+1\right)e^{x+y}\left(-e^{x+y}+e^{x}+e^{y}+3\right)}{2\left(e^{x}+e^{y}+2\right)\left(e^{x+y}-1\right)\left(3e^{x+y}+e^{2x+y}+e^{x+2y}-1\right)},\\
G_{P}(x,y) & =\mathrm{Var}(\frac{\partial l}{\partial y},m_{-l})=2\cdot\frac{\partial S}{\partial y}\cdot\mathbf{p}_{3}(x,y)\\
 & =\frac{\left(e^{x}+1\right)^{2}e^{y}\left(e^{x+2y}+e^{x}+2\right)}{2\left(e^{x}+e^{y}+2\right)\left(e^{x+y}-1\right)\left(3e^{x+y}+e^{2x+y}+e^{x+2y}-1\right)}.
\end{align*}
 where $S(x,y)=-\mathrm{log}\left(\frac{e^{x+y}-1}{e^{x}+e^{y}+2}\right)$
and $l(x,y)=(x,y,S(x,y))\in\mathcal{M}_{\mbox{\ensuremath{\mathcal{G}}}_{2}}^{1}$.
\end{prop}

\begin{proof}
Only $F_{P}(x,y)$ needs some elaboration. By the parallelogram formula
and Proposition \ref{Prop: var formula graph} , we have 
\begin{align*}
F_{P}(x,y) & =\frac{1}{2}\left(||\frac{\partial l}{\partial x}+\frac{\partial l}{\partial y}||_{P}^{2}-||\frac{\partial l}{\partial x}||_{P}^{2}-||\frac{\partial l}{\partial y}||_{P}^{2}\right)\\
 & =(\frac{\partial}{\partial x}+\frac{\partial}{\partial y})^{2}S(x,y)\cdot\mathbf{p}_{3}(x,y)-(\frac{\partial}{\partial x})^{2}S(x,y)\cdot\mathbf{p}_{3}(x,y)-(\frac{\partial}{\partial y})^{2}S(x,y)\cdot\mathbf{p}_{3}(x,y)\\
 & =2\frac{\partial^{2}}{\partial x\partial y}S(x,y)\cdot\mathbf{p}_{3}(x,y).
\end{align*}

\end{proof}

\textsf{We notice that the same computation holds for Weil-Petersson
metric ($||\cdot||_{WP}$). Therefore, we have the following result. }

\begin{prop}
The first fundamental form of $(\mathcal{M}_{\mathcal{G}_{2}}^{1},||\cdot||_{WP})$
is 
\begin{align*}
E_{WP}(x,y) & =\frac{\mathrm{Var}(\frac{\partial l}{\partial x},m_{-l})}{V(l)}\\
 & =\frac{e^{x}\left(e^{y}+1\right)^{2}\left(e^{2x+y}+e^{y}+2\right)}{\left(e^{x}+e^{y}+2\right)\left(e^{x+y}-1\right)f(x,y)},\\
F_{WP}(x,y) & =\frac{\frac{1}{2}\left(\mathrm{Var}(\frac{\partial l}{\partial x}+\frac{\partial l}{\partial y},m_{-l})-\mathrm{Var}(\frac{\partial l}{\partial x},m_{-l})-\mathrm{Var}(\frac{\partial l}{\partial y},m_{-l})\right)}{V(l)}\\
 & =-\frac{\left(e^{x}+1\right)\left(e^{y}+1\right)e^{x+y}\left(e^{x+y}-e^{x}-e^{y}-3\right)}{\left(e^{x}+e^{y}+2\right)\left(e^{x+y}-1\right)f(x,y)},\\
G_{WP}(x,y) & =\frac{\mathrm{Var}(\frac{\partial l}{\partial y},m_{-l})}{V(l)}\\
 & =\frac{\left(e^{x}+1\right)^{2}e^{y}\left(e^{x+2y}+e^{x}+2\right)}{\left(e^{x}+e^{y}+2\right)\left(e^{x+y}-1\right)f(x,y)},
\end{align*}
where $V(l)=\int l\mathrm{d}m_{-l}=2(l_{1}\cdot\mathbf{p}_{1}+l_{2}\cdot\mathbf{p}_{2}+l_{3}\cdot\mathbf{p_{3}})$
and 
\[
f(x,y)=\left(xe^{x+2y}+ye^{2x+y}+2e^{x+y}(x+y)+\left(-2e^{x+y}-e^{2x+y}-e^{x+2y}+e^{x}+e^{y}+2\right)\log\left(\frac{e^{x+y}-1}{e^{x}+e^{y}+2}\right)+e^{x}x+e^{y}y\right)
\]

\end{prop}

\begin{prop}
\label{prop: Belt buckle Pressure}The moduli space of belt buckles
is bounded positively curved under the pressure metric $||\cdot||_{P}$.
i.e., $(\mathcal{M}_{\mathcal{G}_{2}}^{1},||\cdot||_{P})$ is positively
curved and the Gaussian curvature is bounded. Moreover, $(\mathcal{M}_{\mathcal{G}_{2}}^{1},||\cdot||_{P})$
is incomplete.
\end{prop}

\begin{proof}
By the Brioschi formula, we can write down the curvature explicitly
as the following 
\begin{align*}
K_{P}(x,y) & =\frac{1}{4\left(e^{x}+1\right)^{2}\left(e^{y}+1\right)^{2}\left(3e^{x+y}+e^{2x+y}+e^{x+2y}-1\right)}\cdot\left(5+6e^{x}+3e^{2x}+6e^{y}+3e^{2y}+3e^{x+y}+\right.\\
 & \;\;\;\;45e^{2(x+y)}+19e^{3(x+y)}+11e^{2x+y}+9e^{3x+y}+3e^{4x+y}+11e^{x+2y}+\\
 & \;\;\;\;\left.33e^{3x+2y}+8e^{4x+2y}+9e^{x+3y}+33e^{2x+3y}+3e^{4x+3y}+3e^{x+4y}+8e^{2x+4y}+3e^{3x+4y}\right).
\end{align*}
We observe that the numerator and the denominator of $K(x,y)$ have
the same highest exponents $e^{3x+4y}$ and $e^{4x+3y}$. Consider
the polar coordinate $e^{x}=r\cos\theta$ and $e^{y}=r\sin\theta$
where $1<r<\infty$ and $\theta\in(0,\frac{\pi}{2})$ 
\[
K_{P}(x,y)=\frac{6r^{7}(\cos^{3}\theta\sin^{4}\theta+\cos^{4}\theta\sin^{3}\theta)+HOT(r^{6})}{8r^{7}(\cos^{3}\theta\sin^{4}\theta+\cos^{4}\theta\sin^{3}\theta)+HOT(r^{6})}.
\]
 Therefore, we know $K_{P}(x,y)$ is close to $\frac{3}{4}$ when
$r$ is large, and it is easy see that $K_{p}(x,y)>0$ for $x,y\geq0$.
Because $B:=\{(x,y);x\geq0,y\geq0,e^{2x}+e^{2y}\leq r^{2}\}$ is a
compact set, we know $c_{1}\geq K_{p}(x,y)\geq c_{0}>0$.

To prove the second assertion, we consider the path $c(x)=l(x,x)=(x,x,-\log\left(\frac{1-e^{-2x}}{2e^{-2x}+2e^{-x}}\right)).$
We repeat the argument that we used in Proposition \ref{prop: fig 8 press incomplete}.
Because when $x$ goes to $0$, $c(x)$ goes to the boundary of $\mathcal{M}_{\mathcal{G}_{2}}^{1}$,
to prove the incompleteness we only need to show that $c(x)$ goes
to boundary in finite time. Since 
\begin{align*}
||\dot{c}(x)||_{P}^{2}= & P_{3}(x,x)\cdot\frac{d^{2}}{dx^{2}}\left(-\log\left(\frac{1-e^{-2x}}{2e^{-2x}+2e^{-x}}\right)\right)+P_{6}(x,x)\cdot\frac{d^{2}}{dx^{2}}\left(-\log\left(\frac{1-e^{-2x}}{2e^{-2x}+2e^{-x}}\right)\right)\\
= & \frac{e^{x}}{-3e^{x}+2e^{2x}+1},
\end{align*}
it is clear that $\int_{0}^{1}||\dot{c}(x)||_{P}\mathrm{d}x$ is convergent.
Hence the geodesic distance from $c(1)$ to the boundary point $c(0)$
is finite. 
\end{proof}

\begin{obs}

\label{obs: Belt buckle WP}The sectional curvature of $(\mathcal{M}_{\mathcal{G}_{2}}^{1},||\cdot||_{WP})$
is negative and bounded.

\end{obs}

\begin{proof}[Evidence]

By the Brioschi formula, we can explicitly write down the curvature.
However, it is too long and unnecessary to state it in the proof.
We put the explicit expression of $K_{WP}(x,y)$ in the appendix for
whom is interested in that. 

The following is the graph of $K_{WP}(x,y)$ provided $0<x,y<5$ and
the hidden condition $e^{x+y}<3+e^{x}+e^{y}$. With the explicit formula
of $K_{WP}$, we can easily plot the picture and estimate the maximum
and the minimal of $K_{WP}$ through softwares, and here we use Mathematica.
By Mathematica, we know the maximum of $K_{WP}(x,y)$ for $0<x,y$
and $e^{x+y}<3+e^{x}+e^{y}$ happens around $(x_{0},y_{0})\approx(1.09861,1.09861)$
where $K_{WP}(x_{0},y_{0})\approx-0.485025$, and the lower bound
of $K_{WP}$ is around $-0.564958$.

\begin{center}\includegraphics[scale=0.8]{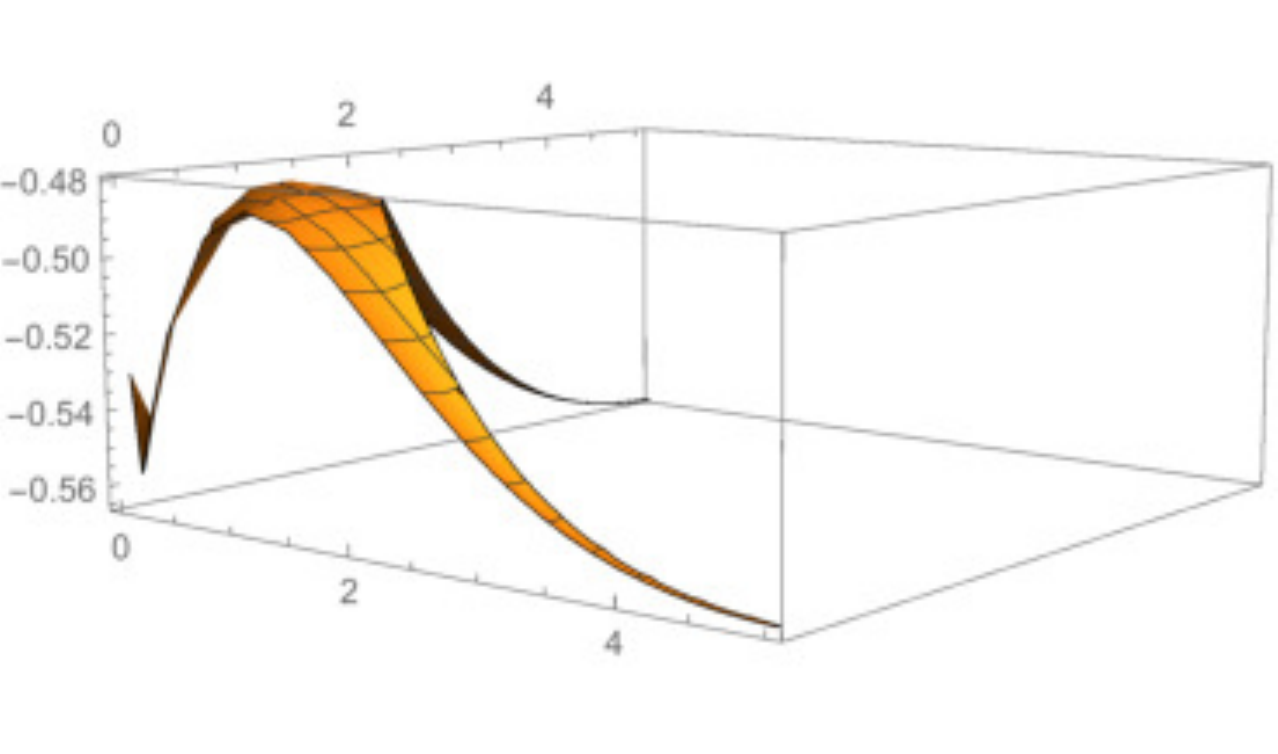}\end{center}

\end{proof}

\subsection{Dumbbells}

\textsf{In this subsection, we discuss the graph of a dumbbell. We
denote this graph by $\mathcal{G}_{3}$ and the picture below is what
it looks like. Likewise, for short, we denote $l(e_{1})$, $l(e_{2})$
and $l(e_{3})$ by $x$, $y$ and $z$, respectively. }

\begin{center}\includegraphics{dumbbell} \end{center}

\textsf{Following the same recipe, we have }

\begin{itemize}

\item$A=\left(\begin{array}{cccccc}
1 & 0 & 1 & 0 & 0 & 0\\
0 & 1 & 0 & 0 & 0 & 1\\
0 & 1 & 0 & 0 & 1 & 0\\
0 & 0 & 1 & 1 & 0 & 0\\
0 & 0 & 0 & 0 & 1 & 1\\
1 & 0 & 0 & 1 & 0 & 0
\end{array}\right)$ $A_{-l}=\left(\begin{array}{cccccc}
e^{-x} & 0 & e^{-x} & 0 & 0 & 0\\
0 & e^{-y} & 0 & 0 & 0 & e^{-y}\\
0 & e^{-z} & 0 & 0 & e^{-z} & 0\\
0 & 0 & e^{-x} & e^{-x} & 0 & 0\\
0 & 0 & 0 & 0 & e^{-y} & e^{-y}\\
e^{-z} & 0 & 0 & e^{-z} & 0 & 0
\end{array}\right).$

\item$h(l)=1$ $\implies$$\det(A_{-l}-\mathrm{Id})=0\implies$$e^{-z}=\sqrt{\frac{e^{-x-y}-e^{-x}-e^{-y}+1}{4e^{-x-y}}}$\textsf{
.}

\item $\mathbf{v}=\left(\frac{\sqrt{\frac{(e^{-x}-1)(e^{-y}-1)}{e^{-x-y}}}}{2-2e^{-x}},\frac{1}{2e^{-y}},-\frac{e^{-x}\sqrt{\frac{(e^{-x}-1)(e^{-y}-1)}{e^{-x-y}}}}{e^{-x}-1},\frac{\sqrt{\frac{(e^{-x}-1)(e^{-y}-1)}{e^{-x-y}}}}{2-2e^{-x}},\frac{1}{2e^{-y}},1\right).$

\item $P=\left(\begin{array}{cccccc}
e^{-x} & 0 & \frac{1}{2} & 0 & 0 & 0\\
0 & e^{-y} & 0 & 0 & 0 & \frac{1}{2}\\
0 & 1-e^{-y} & 0 & 0 & 1-e^{-y} & 0\\
0 & 0 & \frac{1}{2} & a & 0 & 0\\
0 & 0 & 0 & 0 & e^{-y} & \frac{1}{2}\\
1-e^{-x} & 0 & 0 & 1-e^{-x} & 0 & 0
\end{array}\right)$.

\item $\mathbf{p}=\left(\begin{array}{c}
\mathbf{p}_{1}(x,y)\\
\mathbf{p}_{2}(x,y)\\
\mathbf{p}_{3}(x,y)\\
\mathbf{p}_{4}(x,y)\\
\mathbf{p}_{5}(x,y)\\
\mathbf{p}_{6}(x,y)
\end{array}\right)=\left(\begin{array}{c}
\frac{e^{x}\left(-1+e^{y}\right)}{4-6e^{x}-6e^{y}+8e^{x+y}}\\
\frac{e^{y}\left(-1+e^{x}\right)}{4-6e^{x}-6e^{y}+8e^{x+y}}\\
\frac{\left(-1+e^{x}\right)\left(-1+e^{y}\right)}{2-3e^{x}-3e^{y}+4e^{x+y}}\\
\mathbf{p}_{1}(x,y)\\
\mathbf{p}_{2}(x,y)\\
\mathbf{p}_{3}(x,y)
\end{array}\right)$ .

\item\textsf{ Now we consider the surface }$\mathcal{M}_{\mbox{\ensuremath{\mathcal{G}}}_{3}}^{1}\ni l=(l_{1},l_{2},l_{3})=(x,y,\log(2)-\frac{1}{2}\log\left(\left(e^{x}-1\right)\left(e^{y}-1\right)\right))=l(x,y)$. 

\item $\frac{\partial}{\partial x}l=(1,0,\frac{e^{x}}{2-2e^{x}})$,
$\frac{\partial l}{\partial y}=(0,1,\frac{e^{y}}{2-2e^{y}}),$ $\frac{\partial^{2}}{\partial x^{2}}l=(0,0,\frac{e^{x}}{2\left(e^{x}-1\right)^{2}}),$$\frac{\partial^{2}}{\partial x\partial y}l=(0,0,0)$,
\textsf{and} $\frac{\partial^{2}}{\partial y^{2}}l=(0,0,\frac{e^{y}}{2\left(e^{y}-1\right)^{2}}).$

\end{itemize}

\begin{rem}
There is a hidden condition $l_{3}>0$. In this case, the condition
is equivalent to 
\[
4>(e^{x}-1)(e^{y}-1).
\]
 
\end{rem}

\textsf{Set  $x=l_{1}$ and $y=l_{2}$. Since $\mathcal{M}_{\mathcal{G}_{3}}^{1}$
is also a two-dimensional manifold in $\real^{3}$, so we repeat the
same argument as for $\mathcal{M}_{\mathcal{G}_{2}}^{1}$. Following
propositions are coming from direct computations.}

\begin{prop}
The first fundamental form of $(\mathcal{M}_{\mathcal{G}_{3}}^{1},||\cdot||_{P})$
is 
\begin{align*}
E_{P}(x,y) & =\mathrm{Var}(\frac{\partial l}{\partial x},m_{-l})=2\cdot\frac{\partial l}{\partial x}\cdot\mathbf{p}_{3}(x,y)\\
 & =\frac{e^{x}\left(e^{y}-1\right)}{\left(e^{x}-1\right)\left(4e^{x+y}-3e^{x}-3e^{y}+2\right)},\\
F_{P}(x,y) & =\frac{1}{2}\left(\mathrm{Var}(\frac{\partial l}{\partial x}+\frac{\partial l}{\partial y},m_{-l})-\mathrm{Var}(\frac{\partial l}{\partial x},m_{-l})-\mathrm{Var}(\frac{\partial l}{\partial y},m_{-l})\right)\\
 & =0,\\
G_{P}(x,y) & =\mathrm{Var}(\frac{\partial l}{\partial y},m_{-l})=2\cdot\frac{\partial l}{\partial y}\cdot\mathbf{p}_{3}(x,y)\\
 & =\frac{\left(e^{x}-1\right)e^{y}}{\left(e^{y}-1\right)\left(4e^{x+y}-3e^{x}-3e^{y}+2\right)}.
\end{align*}
 
\end{prop}

\begin{prop}
The first fundamental form of $(\mathcal{M}_{\mathcal{G}_{3}}^{1},||\cdot||_{WP})$
is 
\begin{align*}
E_{WP}(x,y) & =\frac{\mathrm{Var}(\frac{\partial l}{\partial x},m_{-l})}{V(l)}\\
 & =\frac{e^{x}\left(e^{y}-1\right)}{\left(e^{x}-1\right)f(x,y)},\\
F_{WP}(x,y) & =\frac{\frac{1}{2}\left(\mathrm{Var}(\frac{\partial l}{\partial x}+\frac{\partial l}{\partial y},m_{-l})-\mathrm{Var}(\frac{\partial l}{\partial x},m_{-l})-\mathrm{Var}(\frac{\partial l}{\partial y},m_{-l})\right)}{V(l)}\\
 & =0,\\
G_{WP}(x,y) & =\frac{\mathrm{Var}(\frac{\partial l}{\partial y},m_{-l})}{V(l)}\\
 & =\frac{\left(e^{x}-1\right)e^{y}}{\left(e^{y}-1\right)f(x,y)},
\end{align*}
where $V(l)=\int l\mathrm{d}m_{-l}=2(l_{1}\cdot\mathbf{p}_{1}+l_{2}\cdot\mathbf{p}_{2}+l_{3}\cdot\mathbf{p_{3}})$
and 
\[
f(x,y)=\left(xe^{x+y}+ye^{x+y}-\log\left(\left(e^{x}-1\right)\left(e^{y}-1\right)\right)+2\left(-e^{x+y}+e^{x}+e^{y}\right)\log\left(\frac{1}{2}\sqrt{\left(e^{x}-1\right)\left(e^{y}-1\right)}\right)-e^{x}x-e^{y}y+\log(4)\right)
\]

\end{prop}

\begin{prop}
\label{prop: Dumbbell Pressure}The moduli space of dumbbells is positively
curved under the pressure metric $||\cdot||_{P}$. i.e., $(\mathcal{M}_{\mathcal{G}_{3}}^{1},||\cdot||_{P})$
is positively curved. More precisely, the Gaussian curvature is strictly
bigger than zero, but has no upper bound. Moreover, $(\mathcal{M}_{\mathcal{G}_{3}}^{1},||\cdot||_{P})$
is incomplete.
\end{prop}

\begin{proof}
Applying the Brioschi formula, we can write down the curvature explicitly:
\[
K_{P}(x,y)=\frac{2e^{x+y}-1}{4e^{x+y}-3e^{x}-3e^{y}+2}
\]
where $x=l_{1}$ and $y=l_{2}$. It is clear that $K_{P}(x,y)>0$
for all $x,y>0$, and ${\displaystyle \lim_{(x,y)\to(0,0)}K_{p}(x,y)=\infty}$.
Moreover, set $a=e^{x}$ and $b=e^{y}$, 
\[
K_{P}(a,b)=\frac{2ab-1}{4ab-3a-3b+2}
\]
provided $a>1$, $b>1$ and $(a-1)(b-1)<4$. Since $\partial_{a}K_{P}(a,b)=0$
and $\partial_{b}K_{P}(a,b)=0$ has no real solution, we know the
extreme values of $K_{P}(a,b)$ could only happen on the boundaries
$a=1$, $b=1$ or $(a-1)(b-1)<4$. Easy computation shows that $K_{P}(a,b)$
is strictly bigger than 0 on these boundaries. Therefore, we can conclude
that $K_{P}(x,y)>c_{0}>0$ for some $c_{0}$. And Mathematica computation
indicates that $c_{0}$ is about 0.85.

To prove the last assertion, we consider the path $c(x)=l(x,x)=(x,x,\log(2)-\frac{1}{2}\log\left(\left(e^{x}-1\right)\left(e^{y}-1\right)\right)).$
We repeat the argument that we used in Proposition \ref{prop: fig 8 press incomplete}.
Because when $x$ goes to $0$, $c(x)$ goes to the boundary of $\mathcal{M}_{\mathcal{G}_{2}}^{1}$,
to prove the incompleteness we only need to show that $c(x)$ goes
to boundary in finite time. Since 
\begin{align*}
||\dot{c}(x)||_{P}^{2}= & P_{3}(x,x)\cdot\frac{d^{2}}{dx^{2}}\left(\log(2)-\log\left(e^{x}-1\right)\right)+P_{6}(x,x)\cdot\frac{d^{2}}{dx^{2}}\left(\log(2)-\log\left(e^{x}-1\right)\right)\\
= & \frac{e^{x}}{-3e^{x}+2e^{2x}+1},
\end{align*}
it is clear that $\int_{0}^{1}||\dot{c}(x)||_{P}\mathrm{d}x$ is convergent.
Hence the geodesic distance from $c(1)$ to the boundary point $c(0)$
is finite. 
\end{proof}

\begin{obs}\label{obs: Dumbbell WP}The sectional curvature of $(\mathcal{M}_{\mathcal{G}_{3}}^{1},||\cdot||_{WP})$
takes positive and negative values.

\end{obs}

\begin{proof}[Evidence]

It is the same as in the belt buckle case, using the Brioschi formula,
we can have the explicit expression of $K_{WP}(x,y)$ which, in the
sake of brevity, is stated the appendix. The following figure is the
graph of $K_{WP}(x,y)$ produced by Mathematica for $0.05<x,y<3$
and. Moreover, Mathematica computation shows that $K_{WP}$ takes
positive (e.g. when $x$ is close to 0.05 and $y$ is close to 3)
and negative values (e.g. when $x$ and $y$ are close to 0.05).

\begin{center}\includegraphics[scale=0.8]{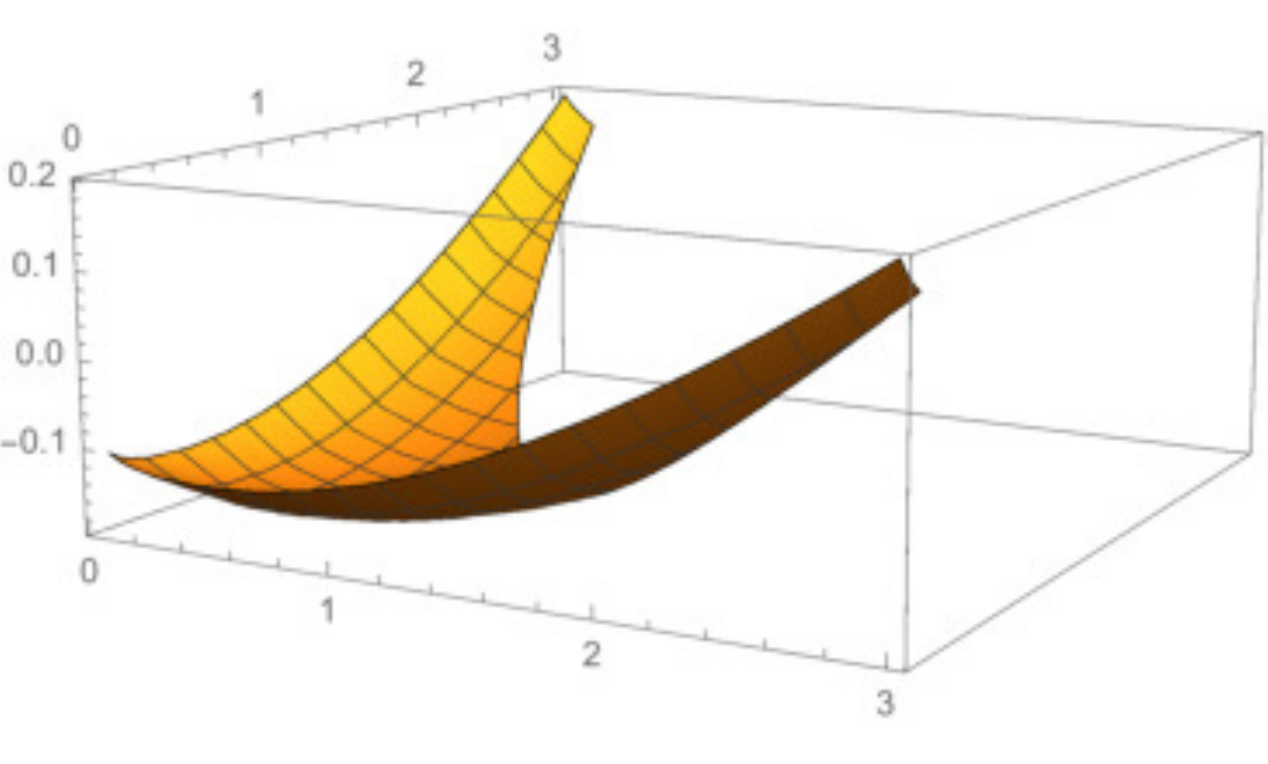}\end{center}

\end{proof}

\subsection{Three-petal roses}

\textsf{The fourth example is a three-petal rose $\mathcal{G}_{4}$,
which could be thought of as a generalization of the figure 8 . The
picture below is a picture of a three-petal rose. For brevity, we
denote $l(e_{1})$, $l(e_{2})$ and $l(e_{3})$ by $x$, $y$ and
$z$, respectively. }

\begin{center}

\includegraphics[scale=0.4]{3-petal_rose}\end{center}

\textsf{Follow the recipe, we have:}

\begin{itemize}

\item$A=\left(\begin{array}{cccccc}
1 & 1 & 1 & 0 & 1 & 1\\
1 & 1 & 1 & 1 & 0 & 1\\
1 & 1 & 1 & 1 & 1 & 0\\
0 & 1 & 1 & 1 & 1 & 1\\
1 & 0 & 1 & 1 & 1 & 1\\
1 & 1 & 0 & 1 & 1 & 1
\end{array}\right)$, $A_{-l}=\left(\begin{array}{cccccc}
e^{-x} & e^{-x} & e^{-x} & 0 & e^{-x} & e^{-x}\\
e^{-y} & e^{-y} & e^{-y} & e^{-y} & 0 & e^{-y}\\
e^{-z} & e^{-z} & e^{-z} & e^{-z} & e^{-z} & 0\\
0 & e^{-x} & e^{-x} & e^{-x} & e^{-x} & e^{-x}\\
e^{-y} & 0 & e^{-y} & e^{-y} & e^{-y} & e^{-y}\\
e^{-z} & e^{-z} & 0 & e^{-z} & e^{-z} & e^{-z}
\end{array}\right).$

\item$h(l)=1$ $\implies$$\det(A_{-l}-Id)=0\implies$$e^{-z}=-\frac{-e^{x+y}+e^{x}+e^{y}+3}{e^{x+y}+3e^{x}+3e^{y}+5}$\textsf{
.}

\item $\mathbf{v}=\left(\frac{2\left(e^{-y}+1\right)}{e^{-x}\left(5e^{-y}+3\right)+3e^{-y}+1},\frac{2\left(e^{-x}+1\right)}{e^{-x}\left(5e^{-y}+3\right)+3e^{-y}+1},1,\frac{2\left(e^{-y}+1\right)}{e^{-x}\left(5e^{-y}+3\right)+3e^{-y}+1},\frac{2\left(e^{-x}+1\right)}{e^{-x}\left(5e^{-y}+3\right)+3e^{-y}+1},1\right).$

\item $P=\left(\begin{array}{cccccc}
\frac{2\left(1+e^{-y}\right)}{e^{-x}\left(3+5e^{-y}\right)+3e^{-y}+1} & 0 & 0 & 0 & 0 & 0\\
0 & \frac{2\left(1+e^{-x}\right)}{e^{-x}\left(3+5e^{-y}\right)+3e^{-y}+1} & 0 & 0 & 0 & 0\\
0 & 0 & 1 & 0 & 0 & 0\\
0 & 0 & 0 & \frac{2\left(1+e^{-y}\right)}{e^{-x}\left(3+5e^{-y}\right)+3e^{-y}+1} & 0 & 0\\
0 & 0 & 0 & 0 & \frac{2\left(1+e^{-x}\right)}{e^{-x}\left(3+5e^{-y}\right)+3e^{-y}+1} & 0\\
0 & 0 & 0 & 0 & 0 & 1
\end{array}\right)$.

\item $\mathbf{p}=\left(\begin{array}{c}
\mathbf{p}_{1}(x,y)\\
\mathbf{p}_{2}(x,y)\\
\mathbf{p}_{3}(x,y)\\
\mathbf{p}_{4}(x,y)\\
\mathbf{p}_{5}(x,y)\\
\mathbf{p}_{6}(x,y)
\end{array}\right)=\left(\begin{array}{c}
\frac{2e^{x}\left(e^{y}+1\right)^{2}}{12e^{x+y}+e^{2(x+y)}+6e^{2x+y}+6e^{x+2y}-10e^{x}-3e^{2x}-10e^{y}-3e^{2y}-15}\\
\frac{2\left(e^{x}+1\right)^{2}e^{y}}{12e^{x+y}+e^{2(x+y)}+6e^{2x+y}+6e^{x+2y}-10e^{x}-3e^{2x}-10e^{y}-3e^{2y}-15}\\
\frac{\left(e^{x+y}-e^{x}-e^{y}-3\right)\left(e^{x+y}+3e^{x}+3e^{y}+5\right)}{2\left(12e^{x+y}+e^{2(x+y)}+6e^{2x+y}+6e^{x+2y}-10e^{x}-3e^{2x}-10e^{y}-3e^{2y}-15\right)}\\
\mathbf{p}_{1}(x,y)\\
\mathbf{p}_{2}(x,y)\\
\mathbf{p}_{3}(x,y)
\end{array}\right)$ .

\item\textsf{ Now we consider the surface }$\mathcal{M}_{\mbox{\ensuremath{\mathcal{G}}}_{4}}^{1}\ni l=(l_{1},l_{2},l_{3})=(x,y,-\text{\ensuremath{\log}}(-\frac{-e^{x+y}+e^{x}+e^{y}+3}{e^{x+y}+3e^{x}+3e^{y}+5}))=l(x,y)$. 

\item $\frac{\partial}{\partial x}l=(1,0,\frac{4e^{x}\left(e^{y}+1\right)^{2}}{\left(-e^{x+y}+e^{x}+e^{y}+3\right)\left(e^{x+y}+3e^{x}+3e^{y}+5\right)})$,
$\frac{\partial l}{\partial y}=(0,1,\frac{4\left(e^{x}+1\right)^{2}e^{y}}{\left(-e^{x+y}+e^{x}+e^{y}+3\right)\left(e^{x+y}+3e^{x}+3e^{y}+5\right)}),$ 

$\frac{\partial^{2}}{\partial x^{2}}l=(0,0,\frac{4e^{x}\left(e^{y}+1\right)^{2}\left(e^{y}+3\right)\left(e^{2x+y}-e^{2x}+3e^{y}+5\right)}{\left(-e^{x+y}+e^{x}+e^{y}+3\right)^{2}\left(e^{x+y}+3e^{x}+3e^{y}+5\right)^{2}}),$$\frac{\partial^{2}}{\partial x\partial y}l=(0,0,\frac{32\left(e^{x}+1\right)\left(e^{y}+1\right)e^{x+y}\left(e^{x}+e^{y}+2\right)}{\left(-e^{x+y}+e^{x}+e^{y}+3\right)^{2}\left(e^{x+y}+3e^{x}+3e^{y}+5\right)^{2}})$,\textsf{ }

\textsf{and} $\frac{\partial^{2}}{\partial y^{2}}l=(0,0,\frac{4\left(e^{x}+1\right)^{2}\left(e^{x}+3\right)e^{y}\left(e^{x+2y}+3e^{x}-e^{2y}+5\right)}{\left(-e^{x+y}+e^{x}+e^{y}+3\right)^{2}\left(e^{x+y}+3e^{x}+3e^{y}+5\right)^{2}}).$

\end{itemize}

\begin{rem}
There is a hidden condition $l_{3}>0$. In this case, the condition
is equivalent to 
\[
-e^{x+y}+e^{x}+e^{y}+3<0.
\]

\end{rem}

\textsf{Since $\mathcal{M}_{\mathcal{G}_{4}}^{1}$ is still a two-dimensional
manifold in $\real^{3}$, so we repeat the same argument as for $\mathcal{M}_{\mathcal{G}_{2}}^{1}$.
Set $x=l_{1}$, $y=l_{2}$ we have the following results of the first
fundamental form with respect to $||\cdot||_{P}$ and $||\cdot||_{WP}$.}

\begin{prop}
The first fundamental form of $(\mathcal{M}_{\mathcal{G}_{4}}^{1},||\cdot||_{P})$
is 
\begin{align*}
E_{P}(x,y) & =\mathrm{Var}(\frac{\partial l}{\partial x},m_{-l})=2\cdot\frac{\partial l}{\partial x}\cdot\mathbf{p}_{3}(x,y)\\
 & =-\frac{4e^{x}\left(e^{y}+1\right)^{2}\left(e^{y}+3\right)\left(e^{2x+y}-e^{2x}+3e^{y}+5\right)}{\left(-e^{x+y}+e^{x}+e^{y}+3\right)\left(e^{x+y}+3e^{x}+3e^{y}+5\right)w(x,y)},\\
F_{P}(x,y) & =\frac{1}{2}\left(\mathrm{Var}(\frac{\partial l}{\partial x}+\frac{\partial l}{\partial y},m_{-l})-\mathrm{Var}(\frac{\partial l}{\partial x},m_{-l})-\mathrm{Var}(\frac{\partial l}{\partial y},m_{-l})\right)\\
 & =-\frac{32\left(e^{x}+1\right)\left(e^{y}+1\right)e^{x+y}\left(e^{x}+e^{y}+2\right)}{\left(-e^{x+y}+e^{x}+e^{y}+3\right)\left(e^{x+y}+3e^{x}+3e^{y}+5\right)w(x,y)},\\
G_{P}(x,y) & =\mathrm{Var}(\frac{\partial l}{\partial y},m_{-l})=2\cdot\frac{\partial l}{\partial y}\cdot\mathbf{p}_{3}(x,y)\\
 & =-\frac{4\left(e^{x}+1\right)^{2}\left(e^{x}+3\right)e^{y}\left(e^{x+2y}+3e^{x}-e^{2y}+5\right)}{\left(-e^{x+y}+e^{x}+e^{y}+3\right)\left(e^{x+y}+3e^{x}+3e^{y}+5\right)w(x,y)}.
\end{align*}
 where $w(x,y)=\left(12e^{x+y}+e^{2(x+y)}+6e^{2x+y}+6e^{x+2y}-10e^{x}-3e^{2x}-10e^{y}-3e^{2y}-15\right)$
\end{prop}

\begin{prop}
The first fundamental form of $(\mathcal{M}_{\mathcal{G}_{4}}^{1},||\cdot||_{WP})$
is 
\begin{align*}
E_{WP}(x,y) & =\frac{\mathrm{Var}(\frac{\partial l}{\partial x},m_{-l})}{V(l)}\\
 & =\frac{4e^{x}\left(e^{y}+1\right)^{2}\left(e^{y}+3\right)\left(e^{2x+y}-e^{2x}+3e^{y}+5\right)}{f_{1}(x,y)f_{2}(x,y)f_{3}(x,y)},\\
F_{WP}(x,y) & =\frac{\frac{1}{2}\left(\mathrm{Var}(\frac{\partial l}{\partial x}+\frac{\partial l}{\partial y},m_{-l})-\mathrm{Var}(\frac{\partial l}{\partial x},m_{-l})-\mathrm{Var}(\frac{\partial l}{\partial y},m_{-l})\right)}{V(l)}\\
 & =\frac{32\left(e^{x}+1\right)\left(e^{y}+1\right)e^{x+y}\left(e^{x}+e^{y}+2\right)}{f_{1}(x,y)f_{2}(x,y)f_{3}(x,y)},\\
G_{WP}(x,y) & =\frac{\mathrm{Var}(\frac{\partial l}{\partial y},m_{-l})}{V(l)}\\
 & =\frac{4\left(e^{x}+1\right)^{2}\left(e^{x}+3\right)e^{y}\left(e^{x+2y}+3e^{x}-e^{2y}+5\right)}{f_{1}(x,y)f_{2}(x,y)f_{3}(x,y)},
\end{align*}
where $V(l)=\int l\mathrm{d}m_{-l}=2(l_{1}\cdot\mathbf{p}_{1}+l_{2}\cdot\mathbf{p}_{2}+l_{3}\cdot\mathbf{p_{3}})$
and 
\begin{alignat*}{1}
f_{1}(x,y) & =-e^{x+y}+e^{x}+e^{y}+3,\\
f_{2}(x,y) & =e^{x+y}+3e^{x}+3e^{y}+5\\
f_{3}(x,y) & =\left[-4\left(xe^{x+2y}+ye^{2x+y}+2e^{x+y}(x+y)+e^{x}x+e^{y}y\right)+\right.\\
 & (-4e^{x+y}+e^{2(x+y)}+2e^{2x+y}+2e^{x+2y}-14e^{x}-3e^{2x}-14e^{y}-3e^{2y}-15)\cdot\\
 & \left.\log\left(-\frac{-e^{x+y}+e^{x}+e^{y}+3}{e^{x+y}+3e^{x}+3e^{y}+5}\right)\right]
\end{alignat*}

\end{prop}

\begin{obs}\label{obs: 3-rose Pressure}The moduli space of three-petal
roses is positively curved under the pressure metric $||\cdot||_{P}$.
i.e., $(\mathcal{M}_{\mathcal{G}_{4}}^{1},||\cdot||_{P})$ is positively
curved and the Gaussian curvature is bounded. Moreover, $(\mathcal{M}_{\mathcal{G}_{4}}^{1},||\cdot||_{P})$
is incomplete.

\end{obs}

\begin{proof}[Evidence]

Applying the Brioschi formula, we can write down the curvature explicitly;
however, in this note we only give the figure of the curvature and
avoid stating lengthy results.

\begin{center}

\includegraphics[scale=0.5]{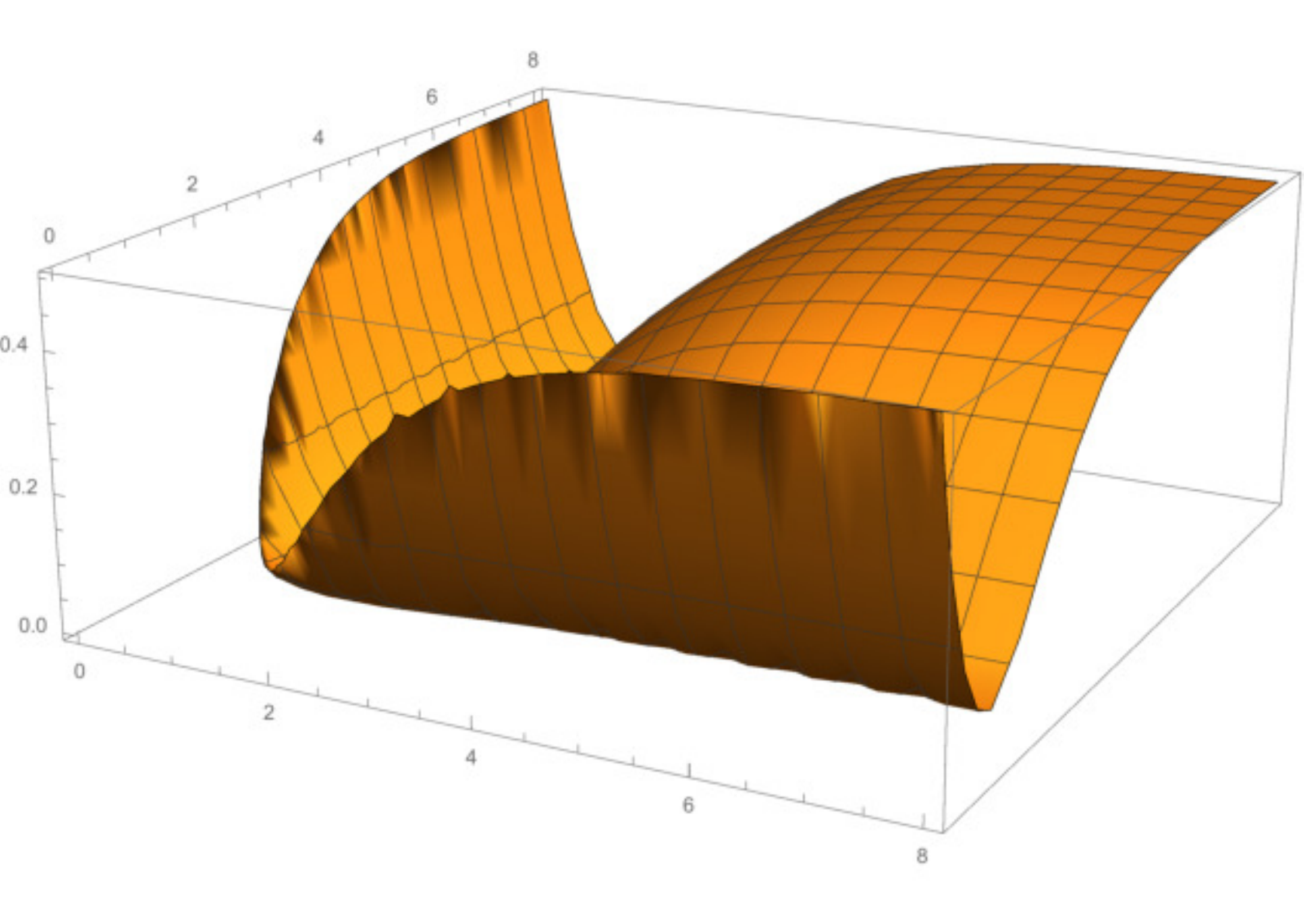}

\end{center}

Moreover, numerical results indicate that $0.2<K_{P}(x,y)<1$. 

The second assertion is because each three-petal rose contains a figure
8, and we know figure 8 is incomplete with respect to the pressure
metric. 

\end{proof}

\begin{obs}

\label{obs: 3-rose-WP}The sectional curvature of $(\mathcal{M}_{\mathcal{G}_{4}}^{1},||\cdot||_{WP})$
takes positive and negative values.

\end{obs}

\begin{proof}[Evidence]

The explicit formula of $K_{WP}$ is complex and make no sense to
state in here. The following figure of $K_{WP}$ for $0.5<x,y<20$
is produced by Mathematica , which indicates that $K_{WP}$ is positive
when $(x,y)=(5,15)$ and negative when $(x,y)=(19,19)$. 

\begin{center}

\includegraphics[scale=0.5]{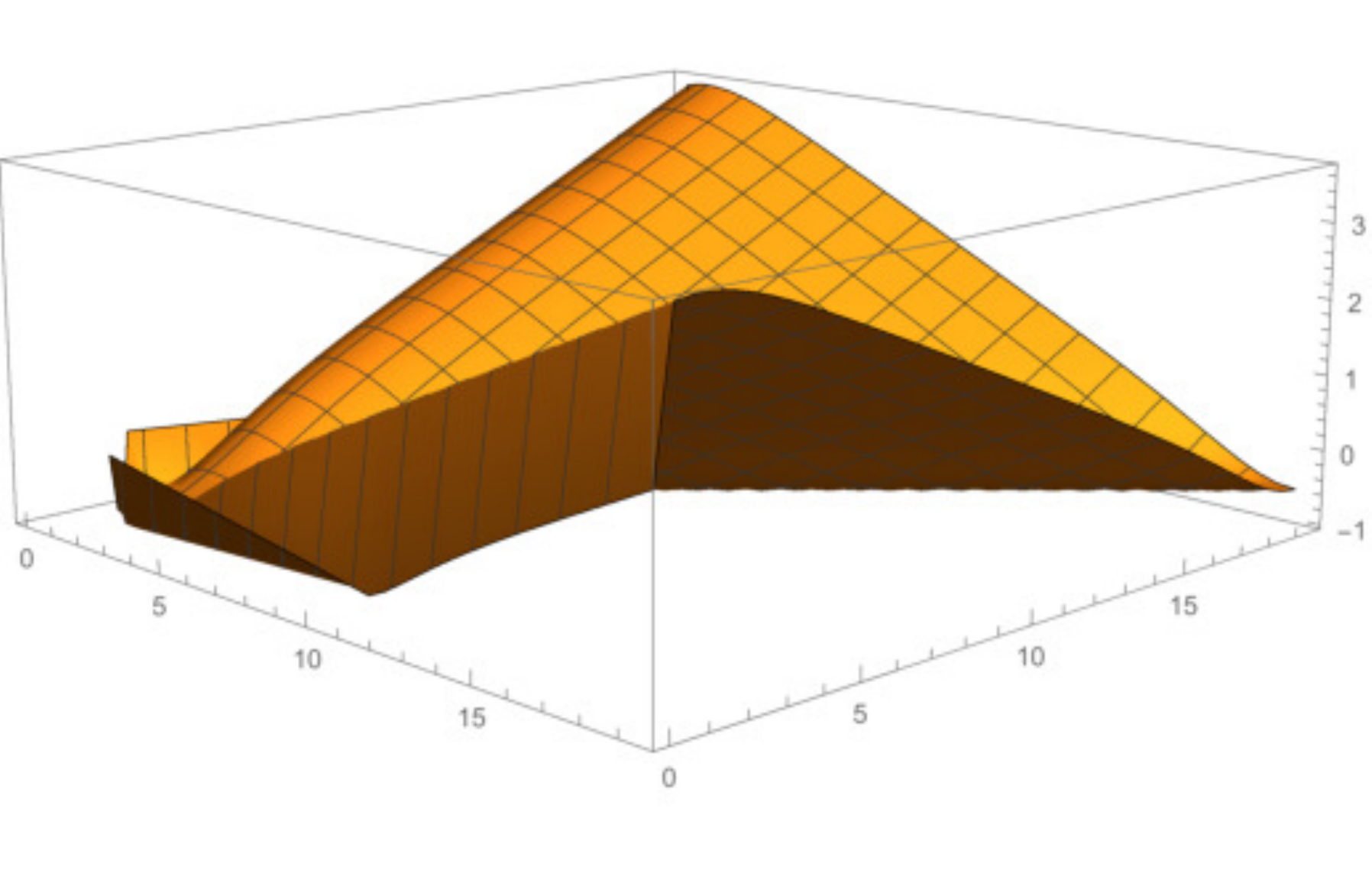}

\end{center}

\end{proof}

\newpage{}

\section{Appendix}

\begin{prop}
The curvature of $(\mathcal{M}_{\mathcal{G}_{2}}^{1},||\cdot||_{WP})$
could be written explicitly as the following. For $0<x,y$ and $e^{x+y}<3+e^{x}+e^{y}$,
we have 

\begin{align*}
 & K_{WP}(x,y)=\\
 & \frac{1}{4\left(e^{x}+1\right)^{2}\left(e^{y}+1\right)^{2}f(x,y)}\cdot\left(e^{x}x(x-y-4)+2e^{2x}x(x-y-4)+2x(y-2)e^{3x+4y}+2(x-2)ye^{4x+3y}-\right.\\
 & \;(x+4)e^{3y}y-e^{3x}x(y+4)+e^{y}y(-x+y-4)+2e^{2y}y(-x+y-4)+ye^{4x+y}(2x+y-4)+\\
 & \;2ye^{4x+2y}(2x+y-4)+xe^{x+4y}(x+2y-4)+2xe^{2x+4y}(x+2y-4)+4e^{x+y}\left(x^{2}-4x+(y-4)y\right)+\\
 & \;2e^{x+3y}\left(2x^{2}+x(3y-8)-8y\right)+2e^{3x+y}(x(3y-8)+2(y-4)y)+4e^{3(x+y)}(x(3y-4)-4y)+\\
 & \;e^{2x+y}\left(8x^{2}+x(5y-32)+6(y-4)y\right)+e^{x+2y}\left(6x^{2}+x(5y-24)+8(y-4)y\right)+\\
 & \;4e^{2(x+y)}\left(3x^{2}+x(5y-12)+3(y-4)y\right)+e^{2x+3y}\left(8x^{2}+x(17y-32)-24y\right)+e^{3x+2y}(x(17y-24)+8(y-4)y)-\\
 & \;\left(5(x+4)e^{y}+4(x+4)e^{2y}+(x+4)e^{3y}+4e^{2x+4y}(x+y-2)+5e^{x}(y+4)+4e^{2x}(y+4)+e^{3x}(y+4)+\right.\\
 & \;4e^{4x+2y}(x+y-2)+2(x+y+4)+2e^{x+y}(x+y+20)+e^{4x+y}(2x+y-4)+e^{4x+3y}(2x+y-4)+4e^{3x+y}(2x+y-2)+\\
 & \;e^{x+4y}(x+2y-4)+e^{3x+4y}(x+2y-4)+4e^{x+3y}(x+2y-2)+2e^{3(x+y)}(5x+5y-12)+e^{2x+y}(11x+5y+16)+\\
 & \;\left.e^{x+2y}(5x+11y+16)+e^{2x+3y}(19x+17y-32)+e^{3x+2y}(17x+19y-32)+e^{2(x+y)}(26x+26y-24)\right)\\
 & \;\left.\left(e^{x+y}-1\right)^{2}\left(6e^{x+y}+e^{2x+y}+e^{x+2y}+7e^{x}+2e^{2x}+7e^{y}+2e^{2y}+6\right)\log^{2}\left(\frac{e^{x+y}-1}{e^{x}+e^{y}+2}\right)+\log\left(\frac{e^{x+y}-1}{e^{x}+e^{y}+2}\right)\right)
\end{align*}
where 
\[
f(x,y)=\left(xe^{x+2y}+ye^{2x+y}+2e^{x+y}(x+y)+\left(-2e^{x+y}-e^{2x+y}-e^{x+2y}+e^{x}+e^{y}+2\right)\log\left(\frac{e^{x+y}-1}{e^{x}+e^{y}+2}\right)+e^{x}x+e^{y}y\right).
\]

\end{prop}

\begin{prop}
The curvature of $(\mathcal{M}_{\mathcal{G}_{3}}^{1},||\cdot||_{WP})$
could be written explicitly as the following. For $0<x,y$ and $4>(e^{x}-1)(e^{y}-1)$,
we have 
\begin{align*}
 & K_{WP}(x,y)=\\
 & -\left\{ 4\left(e^{x}-1\right)\left(e^{y}-1\right)\right.\cdot\\
 & \left.\left(xe^{x+y}+ye^{x+y}-\log\left(\left(e^{x}-1\right)\left(e^{y}-1\right)\right)+2\left(-e^{x+y}+e^{x}+e^{y}\right)\log\left(\frac{1}{2}\sqrt{\left(e^{x}-1\right)\left(e^{y}-1\right)}\right)-e^{x}x-e^{y}y+\log(4)\right)\right\} ^{-1}\cdot\\
 & \left\{ -4e^{x}x+4e^{2x}x+8xe^{x+y}+4xe^{2(x+y)}-8xe^{2x+y}-4xe^{x+2y}+e^{x}x^{2}-e^{2x}x^{2}-2x^{2}e^{x+y}-x^{2}e^{2(x+y)}+2x^{2}e^{2x+y}+\right.\\
 & x^{2}e^{x+2y}-4e^{y}y+4e^{2y}y+8ye^{x+y}+4ye^{2(x+y)}-4ye^{2x+y}-8ye^{x+2y}-e^{x}xy+e^{2x}xy-xe^{y}y+xe^{2y}y-\\
 & -2y^{2}e^{x+y}-y^{2}e^{2(x+y)}+y^{2}e^{2x+y}+2y^{2}e^{x+2y}+2xye^{2(x+y)}+xye^{2x+y}+xye^{x+2y}+e^{y}y^{2}-e^{2y}y^{2}-\\
 & x\log(4)+xe^{y}\log(4)-y\log(4)+e^{x}y\log(4)-\left(x\left(e^{y}-1\right)+\left(e^{x}-1\right)y\right)\log\left(\left(e^{x}-1\right)\left(e^{y}-1\right)\right)+\\
 & 2\left[-2(x-4)e^{2x+y}-16e^{x+y}+(3x+8)e^{y}+e^{2x}(x-y-4)-2(y-4)e^{x+2y}+e^{2y}(-x+y-4)\right.\\
 & \left.\left.\left(+e^{2(x+y)}(x+y-4)-2(x+y+2)+e^{x}(3y+8)\right)\right]\log\left(\frac{1}{2}\sqrt{\left(e^{x}-1\right)\left(e^{y}-1\right)}\right)\right\} 
\end{align*}

\end{prop}

\bibliographystyle{amsalpha}
\bibliography{/Users/nyima/Dropbox/TEX/Bibtex/Papers3_backup_20151007}

\end{document}